\documentclass[11pt,reqno]{amsart}

\usepackage{amsmath,amssymb,amsfonts}

\usepackage[left=30mm,right=30mm,top=30mm,bottom=30mm]{geometry}
\usepackage{mathrsfs}
\usepackage{hyperref}
\usepackage{cite}
\usepackage[foot]{amsaddr}
\usepackage[sc]{mathpazo} 
\usepackage{graphics}
\usepackage{graphicx}
\usepackage{empheq} 
\usepackage{cases}
\usepackage{enumitem}
\usepackage{comment}

\newcommand{\Z}{\mathbb{Z}}
\newcommand{\N}{\mathbb{N}}
\newcommand{\R}{\mathbb{R}}
\newcommand{\C}{\mathbb{C}}
\renewcommand{\L}{\LL^2}
\newcommand{\LL}{\mathsf{L}}
\renewcommand{\H}{\mathsf{H}}
\newcommand{\HS}{\mathscr{H}}

\renewcommand{\a}{\mathfrak{a}}
\newcommand{\A}{{\mathscr{A}}}

\newcommand{\Res}{{\mathscr{R}}}
\newcommand{\J}{\mathscr{J}} 

\newcommand{\B}{\mathscr{B}}

\renewcommand{\rho}{\varrho}

\DeclareMathOperator{\dom}{dom}
\DeclareMathOperator{\supp}{supp}
\DeclareMathOperator{\dist}{dist}
\DeclareMathOperator{\diam}{diam}
\DeclareMathOperator{\capty}{\mathscr{C}}
\DeclareMathOperator{\area}{area}
\DeclareMathOperator{\vol}{vol}
\DeclareMathOperator{\conv}{conv}

\newcommand{\la}{\langle}
\newcommand{\ra}{\rangle}

\newcommand{\eps}{\varepsilon}
\newcommand{\e}{_{\varepsilon}}
\newcommand{\ke}{_{k,\varepsilon}}

\newcommand{\al}{\alpha}

\newcommand{\ga}{\gamma}

\renewcommand{\d}{\,\mathrm{d}}
\newcommand{\ds}{\displaystyle}

\newcommand{\Id}{\mathrm{I}}

\newcommand{\cupl}{\bigcup\limits}
\newcommand{\suml}{\sum\limits}

\newcommand{\supl}{\sup\limits}

\newcommand{\wt}{\widetilde}
\newcommand{\wh}{\widehat}

\newcommand{\ceq}{\coloneqq}

\newcommand{\restr}{\restriction}

\theoremstyle{plain}
\newtheorem{theorem}{Theorem}[section]
\newtheorem*{theorem*}{Theorem}
\newtheorem{lemma}[theorem]{Lemma}
\newtheorem*{lemma*}{Lemma}
\newtheorem{proposition}[theorem]{Proposition}
\newtheorem{corollary}[theorem]{Corollary}
\theoremstyle{remark}
\newtheorem{remark}[theorem]{Remark}
\newtheorem*{remark*}{Remark} 

\newtheorem*{example*}{Example} 
\theoremstyle{definition}

\numberwithin{equation}{section}
\numberwithin{figure}{section}

\title
[Operator estimates for Neumann sieve problem]
{Operator estimates for Neumann sieve problem}

\author[Andrii Khrabustovskyi]{Andrii Khrabustovskyi\,$^{1,2}$}
\address{$^1$ Department of Physics, Faculty of Science, University of
	Hradec Kr\'{a}lov\'{e}, Czech Republic} 
\address{$^2$ Department of Theoretical Physics,
	Nuclear Physics Institute of the Czech Academy of Sciences, \v{R}e\v{z}, Czech Republic} 
\email{andrii.khrabustovskyi@uhk.cz}

\begin{document}
	
	\begin{abstract}
		Let $\Omega$ be a domain in $\mathbb{R}^n$, 
		$\Gamma$ be a hyperplane intersecting $\Omega$, $\varepsilon>0$ be a small parameter, and
		$D_{k,\varepsilon}$, $k=1,2,3\dots$ be a family of small ``holes'' in $\Gamma\cap\Omega$; when $\varepsilon \to 0$, the number of holes  tends to infinity, while 
		their diameters tends to zero. Let $\mathscr{A}_\varepsilon$ be the Neumann Laplacian in the perforated domain $\Omega_\varepsilon=\Omega\setminus\Gamma_\varepsilon$, where
		$\Gamma_\varepsilon=\Gamma\setminus (\cup_k D_{k,\varepsilon})$ (``sieve'').
		It is well-known that if the sizes of holes are carefully chosen, $\mathscr{A}_\varepsilon$ converges in the strong resolvent sense to the Laplacian on $\Omega\setminus\Gamma$ subject to the so-called $\delta'$-conditions on $\Gamma$.
		In the current work we improve this result: under rather general assumptions on 
		the shapes and locations of the holes we derive estimates on the rate of convergence in terms of $L^2\to L^2$ and $L^2\to H^1$ operator norms; in the latter case a special corrector is required.  
	\end{abstract}
	
	
	
	\maketitle

	\section{Introduction}
	
	One of the main directions in homogenization theory concerns  boundary value problems in perforated domains of the form $\Omega\e=\Omega\setminus\Gamma\e$, where $\Omega\subset\R^n$ is fixed, while the  set $\Gamma\e$  depends on a small parameter $\eps>0$ and its geometry is getting more and more complicated as $\eps\to 0$. Homogenization theory is aimed to construct an effective (homogenized) boundary value problem in the unperturbed domain $\Omega$ such that its solutions approximate  solutions of the initial problem in $\Omega\e$ as $\eps\to 0$.
	
	The typical example of such a domain $\Omega\e$ is the one with a lot of tiny holes, that is $\Gamma\e=\cup_{k}D\ke$, where $D\ke\subset\Omega$ are pairwise disjoint compact sets whose  
	diameters  tends to zero as $\eps\to 0$, while their number (per finite volume)  tends to infinity. The holes $D\ke$ may be distributed in the entire volume or  along some hypersurface.  
	
	In the current work we consider the perforation $\Gamma\e$ of another type -- the one having the form of a ``sieve'', and revisit the so-called 
	\emph{Neumann sieve problem} studying Neumann Laplacian in such a perforated domain $\Omega\e\setminus\Gamma\e$. In the next subsection we recall the problem setting  and known results.

	\subsection{Neumann sieve problem}\label{subsec:11}
	
	Let $\Omega$ be a domain in $\R^n$, and
	$\Gamma$ be a hyperplane intersecting 
	$\Omega$ and dividing it on two subsets $\Omega^+$ and $\Omega^-$.
	Further, let $\eps>0$ be a small parameter, and  $\{D\ke,\ k=1,2,3\dots\}$ be a family of 
	small subsets of $\Gamma\cap\Omega$ (``holes''). 
	We set 
	$$
	\Omega\e=\Omega^+\cup\Omega^-\cup \left(\cupl_{k\in\N}D\ke\right)=\Omega\setminus\Gamma\e,\quad\text{where }
	\Gamma\e=\Gamma\setminus\left(\cupl_k D\ke\right)\ \text{(``sieve'')}.
	$$
	Since the set $\Gamma\e$ has  Lebesgue measure zero,  
	the Hilbert spaces $\L(\Omega\e)$ and $\L(\Omega)$  coincide. 
	In the following, we will use  {the same notation $\HS$} for both these spaces.

	\begin{figure}[t]
		\begin{picture}(200,140)
			\includegraphics{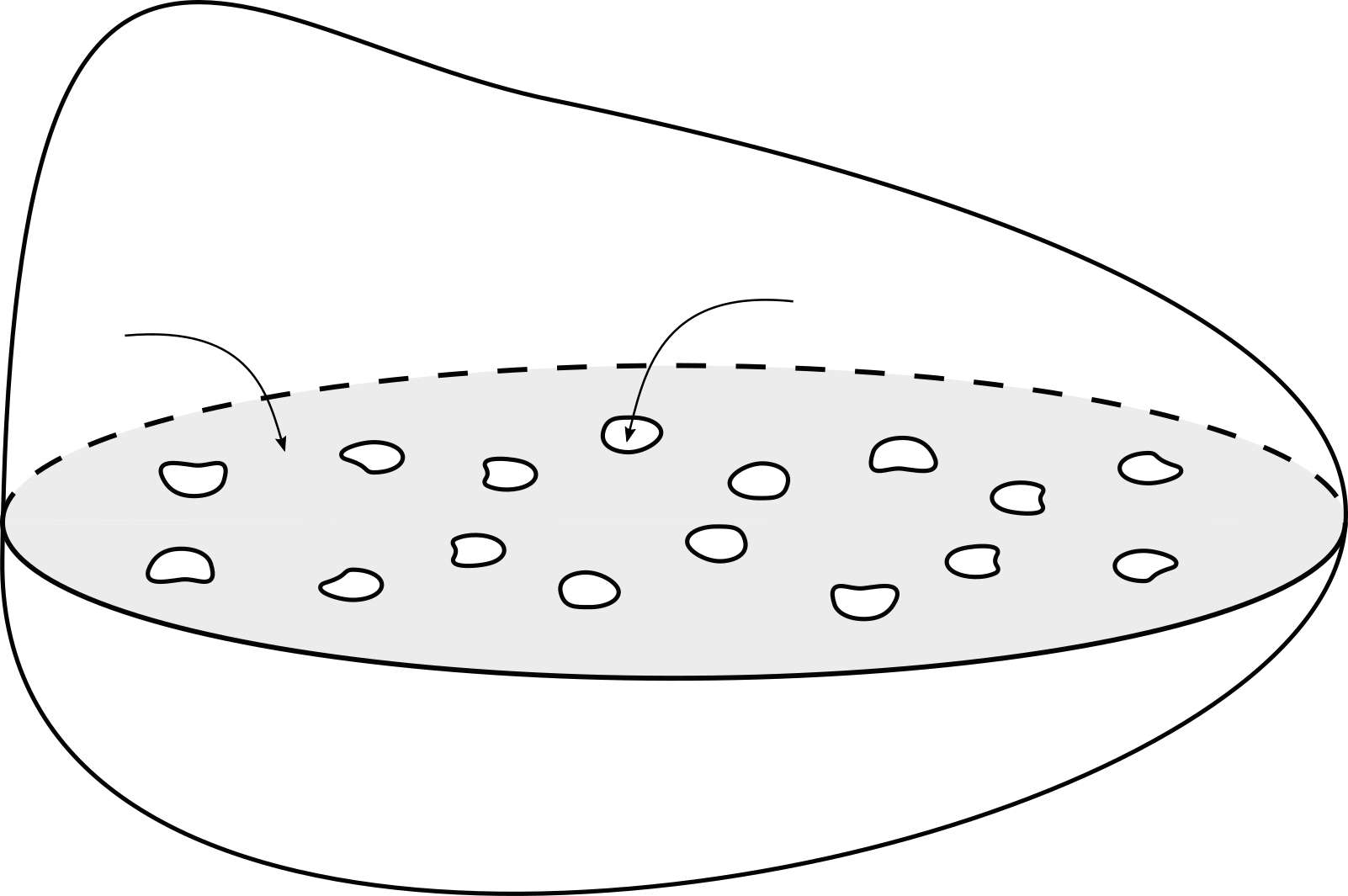}
			
			\put(-160,112){$\Omega^+$}
			\put(-160,12){$\Omega^-$}
			\put(-187,78){$\Gamma\e$}
			\put(-78,82){$D\ke$}
			
		\end{picture}
		
		\caption{The domain $\Omega\e$}\label{fig1}
	\end{figure} 
	
	Let $f\in\HS$.
	We consider the following problem in $\Omega\e$:
	\begin{gather}\label{BVPe}
		\text{Find }u\e\in\H^1(\Omega\e):\qquad
		( u\e, v)_{\H^1(\Omega\e)}=(f,v)_{\HS},\quad
		\forall v\in\H^1(\Omega\e).
	\end{gather}
	It is well-known that this problem has the unique solution $u\e$, and  the mapping
	$\Res\e:f\mapsto u\e $  defines a
	bounded self-adjoint  invertible operator in $\HS$. The operator $\A\e=\Res\e^{-1}-\Id$ is called the \emph{Neumann Laplacian} in $\Omega\e$ (hereinafter, the notation $\Id$ stands either for the identity operator or the identity matrix).
	
	Homogenization theory is aimed to describe the asymptotic behavior of the solution $u\e=(\A\e+\Id)^{-1}f$ 
	to the problem \eqref{BVPe} as $\eps\to 0$, when the number of holes tends to infinity, while 
	their diameters tends to zero. 
	
	Assume for simplicity that the holes $D\ke$ are identical (up to a rigid motion) and  distributed periodically along the $\eps$-periodic lattice on $\Gamma$.  It is well-known \cite{At84,AP87,Mu85,Pi87,MS66,Da85,Ono06,CDGO08,MK74,Kh71} (for more details see the overview of  existing literature in Subsection~\ref{subsec:13}) that if the size of the holes is carefully chosen, namely,  one has
	\begin{gather}\label{G}
		\exists\lim_{\eps\to 0}{\capty\e\over  4\eps^{n-1}}=\gamma\ge 0,\quad\text{where }
		\capty\e\text{ is the capacity of }D\ke,
	\end{gather}
	then 
	\begin{gather}\label{L2strong}
		\|u\e-u\|_{\HS}\to 0\text{ as }\eps\to 0,
	\end{gather}
	and the limiting function $u\in\H^1(\Omega\setminus\Gamma)$ is the unique solution to the following problem:
	\begin{gather}\label{BVPh}
		\text{Find }u\in\H^1(\Omega\setminus\Gamma):\quad
		(u ,v)_{\H^1(\Omega\setminus\Gamma)}+(\gamma [u] ,[v])_{\L(\Gamma\cap\Omega)} =(f,v)_{\HS},\ 
		\forall v\in \H^1(\Omega\setminus\Gamma).
	\end{gather}
	Here $[\cdot]$ stands for the jump of   enclosed function across $\Gamma$.
	It is easy to see that 
	$u$ is a weak solution to the boundary value problem  
	\begin{align*}
		-\Delta u^++ u^+=f^+&\text{\quad on }\Omega^+,\\
		-\Delta u^-+ u^-=f^-&\text{\quad on }\Omega^-,\\
		- {\partial u^+\over\partial  {\nu^+}}   
		=
	{\partial u^-\over\partial  {\nu^-}}
		=
		\gamma(u^+-u^-)
		&
		\text{\quad on }\Gamma\cap\Omega,\\
		{\partial u \over\partial  \nu}=0&\text{\quad on }\partial\Omega,
	\end{align*} 
	where    $u^\pm=u\restriction_{\Omega^\pm}$, $f^\pm=f\restriction_{\Omega^\pm}$, 
	${\partial  \over\partial   \nu^\pm }$ is the derivative along the outward (with respect to $\Omega^\pm$) pointing normal to $\Gamma$, 
	${\partial_\nu}$ is the derivative along the unit normal to $\partial\Omega$.
	We define the homogenized operator $\A$ by $\A=\Res^{-1}-\Id$, where the bounded invertable operator  $\Res$ in $\HS$ is given by
	$ \Res:f\mapsto u  $. The operator is self-adjoint and non-negative. 
	In the recent literature (see, e.g., \cite{BEL14,BLL13}) the operator $\A$ is  called 
	\emph{Schr\"{o}dinger with $\delta'$-interaction of the strength $\gamma$}.
	
	Similar result  holds if $\Gamma$ is  a closed smooth surface:
	under suitable assumptions on the holes $D\ke\subset\Gamma$, which resemble \eqref{G},
	the homogenized problem has the form \eqref{BVPh} with 
	some non-negative and, in general, non-constant  function $\gamma$.

	\subsection{Main results}\label{subsec:12}
	
	In the aforementioned papers the  convergence   \eqref{L2strong} was established for a fixed   $f\in\HS$.  
	In the language of operator theory this means that $\A\e$ converges to the operator $\A$
	as $\eps\to 0$ in \emph{the strong resolvent sense}.
	In the current work we are aimed to upgrade \eqref{L2strong} to
	\emph{the norm resolvent convergence} and derive estimates on its rate. In the homogenization literature (see the overview below) they  are usually called \emph{operator estimates}.
	
	Our first result (see Theorem~\ref{th1}) is the estimate 
	\begin{gather*}
		\|u\e-u\|_{\HS}\leq  \mu\e\|f\|_{\HS}
	\end{gather*}
	with some $\mu\e\to 0$.  An immediate consequence of this result is a convergence of spectra of the underlying operators in the Hausdorff metrics (see Corollary~\ref{th3}).
	
	Secondly, besides the convergence in the 
	$\HS\to \HS$ operator norm,  we also establish the convergence in the
	$\HS\to \H^1(\Omega\setminus\Gamma)$ operator norm.
	In this case, to get a reasonable result one  needs  a  special corrector; the corresponding estimate reads   (see Theorem~\ref{th2}):
	\begin{gather*}
		\|u\e-u-w\e^\pm\|_{\H^1(\Omega^\pm)}\leq \mu\e\|f\|_{\HS}.
	\end{gather*}
	The correctors $w\e^\pm\in\H^1(\Omega^\pm)$ are supported in a small neighborhood 
	of $\Gamma$ and satisfy 
	\begin{gather*}
		\|w\e^\pm\|_{\L(\Omega^\pm)}=o(\mu\e)\|f\|_{\HS},\quad
		\|\nabla w\e^\pm\|^2_{\L(\Omega^\pm)}=
		{1\over 2}\|\ga^{1/2}[u]\|^2_{\L(\Gamma)}+\mathcal{O}(\kappa\e)\|f\|^2_\HS,
	\end{gather*}
	with some $\kappa\e\in [0, \mu\e]$.
	
	The assumptions we impose on the 
	shapes and locations of the holes are rather general (i.e., beyond periodic setting). In the last section we present two examples 
	for which these assumptions are indeed satisfied.
	
	We restrict ourselves to the case when  $\Omega$ is  unbounded  and the hyperplane $\Gamma$ belongs to $\Omega$, moreover $\dist(\Gamma,\partial\Omega)>0$. For example, $\Omega$ may coincide with the whole $\R^n$ or $\Omega$ may be an unbounded waveguide-like domain.
	The case when $\Gamma$ intersects $\partial\Omega$ is not 
	considered in this work (except for the case, when the holes are located within a subset 
	$\Gamma'\subset\Gamma\cap\Omega$ such that $\dist(\Gamma',\partial\Omega)>0$ -- we briefly comment on it in Remark~\ref{rem:Gamma'}). The technical difficulties arising when the holes approaches $\partial\Omega\cap\Gamma$ are explained in Remark~\ref{rem:intersection}.

	Our proof is based on the abstract result from \cite{AP21},
	which is a part of the  toolbox developed in \cite{P06,P12} for studying resolvent and spectral convergence of operators in varying Hilbert spaces.
	Originally this toolbox was developed to handle convergence of the Laplace-Beltrami operator on  manifolds which shrink to a graph. It also has shown to be effective for homogenization problems in perforated domains, see \cite{KP18,AP21,KP21}. 
	The main ingredient of the proof is a construction of  a suitable operator from $\H^1(\Omega\setminus\Gamma)$ to $\H^1(\Omega\e)$ (these spaces are the domains of the sesquilinear forms associated with the operators $\A\e$ and $\A$).

	\subsection{Literature}\label{subsec:13}
	
	\subsubsection{Sieve problem}
	Neumann sieve problem was treated first   in \cite{MS66}.
	In this article $\Gamma$ is a closed surface, and the holes $D\ke\subset\Gamma$ obey quite general assumptions (though, resembling \eqref{G}). It was   proven that the Green's function associated 
	with the operator $\A\e$ converges pointwise to the Green's function associated 
	with the operator $\A$ as $\eps\to 0$. The proof in \cite{MS66} is based on methods of potential theory. Later, in \cite{Kh71}   the convergence of the resolvents of $\A\e$ has been proven
	by using variational methods; in this work the author considered sieves with even more general geometry. 
	Both articles \cite{Kh71,MS66} deal with $n=3$. For more details we refer to the monograph \cite[Chapter~3]{MK74}.
	
	In the mid 80's the problem was revisited 
	in \cite{At84,AP87,Mu85,Pi87,Da85}.
	In \cite{At84,Mu85,Pi87,Da85} the sieve $\Gamma\e$ is flat (a subset of a hyperplane) and the holes are distributed periodically, while in \cite{AP87} $\Gamma$
	is a closed surface, and the holes $D\ke\subset\Gamma$ satisfy some general assumptions.
	The proofs in the above papers are based on variational methods.
	We also mention later articles  \cite{Ono06,CDGO08}, where periodic Neumann sieve problem was studied via  unfolding method;  in particular, in \cite{CDGO08} the authors treated not only   Laplacians, but also more general  elliptic operators. 
	
	Finally, we refer to the articles \cite{DelV87,An04} dealing with a sieve having non-zero thickness, i.e. $\Gamma\e$ is thin layer perforated by a lot of narrow passages connecting the upper and lower faces of this layer.
	
	\subsubsection{Operator estimates}
	Derivation of operator estimates  is a rather young topic in homogenization theory originating from the pioneer works  \cite{BS04,BS06,Gr04,Z05a,Z05b,ZP05,Gr06} (see also the overview \cite{ZP16}) concerning  homogenization of periodic elliptic operators with rapidly oscillating coefficients.
	
	In the last decade there appeared a plenty of works where operator estimates were established
	for homogenization problems in domains with a lot of holes.
	The case when holes are distributed in the entire domain was addressed in 
	\cite{Su18,ZP16,ZP05} 
	(Neumann conditions on the boundary of holes),
	\cite{KP18,AP21} (Dirichlet conditions, also \cite{AP21} treats Neumann holes), 
	\cite{KP21} (linear Robin conditions),
	\cite{BK22,Bo22} (Dirichlet and nonlinear Robin conditions).
	In \cite{Su18,Z05a,KP21} the holes are distributed periodically (in \cite{KP18} -- locally periodically), and are identical, while 
	in \cite{AP21,BK22,Bo22} quite general assumptions on sizes and location of holes are imposed.
	The surface distribution of holes was treated in \cite{BCD16,BM21,BM22,GPS13}.
	
	We also mention the closely related works \cite{BC09,BBC10,BBC13,CDD20} and \cite{BCFP13}, where operator estimates were deduced, respectively, for elliptic operators with frequently alternating boundary conditions and for boundary value problems in domains with fast oscillating boundary.

	\section{Setting of the problem and main results}
	\label{sec2}
	
	Let $n\in\N$, $n\ge 2$.
	In the following, 
	$x'=(x^1,\dots,x^{n-1})$ and $x=(x',x^n)$ stand for   the Cartesian coordinates in $\R^{n-1}$ and $\R^{n}$, respectively. Also, by $\B(r,z)$ we denote
	the open ball in $\R^n$ of radius $r>0$ and center $z\in\R^n$.
	By $C,C_1,C_2,\dots$ we denote generic positive constants being independent of $\eps$; note that these constants may   vary from line to line.

	Let $\Gamma$ be a hyperplane given by
	$$
	\Gamma\ceq  \left\{x=(x',x^n)\in\R^n:\ x^n=0\right\},
	$$
	and $\Omega$ be an unbounded Lipschitz domain satisfying
	\begin{gather}\label{h}
		\Gamma\subset\Omega,\quad  \dist(\Gamma,\partial\Omega)>0
	\end{gather}
	We set
	$$
	\Omega^\pm  \ceq\Omega\cap \left\{x=(x',x^n)\in\R^n:\ \pm x^n>0\right\}.
	$$
	We connect $\Omega^+$ and $\Omega^-$ by making a lot of holes in $\Gamma$.
	Let $\eps\in (0,\eps_0]$ be a small parameter. 
	Let 
	$\left\{D\ke\subset\Gamma,\ k\in \N\right\}$ be a family of  
	relatively open in $\Gamma$ connected  sets; below we will   make further assumptions on their locations and sizes. 
	We define the ``sieve'' $\Gamma\e$ by
	\begin{gather*}
		\Gamma\e\ceq\Gamma\setminus\left(\cupl_{k\in\N}D\ke\right).
	\end{gather*}
	The resulting domain $\Omega\e$ is given by
	$$\Omega\e\ceq\Omega^+\cup\Omega^-\cup \left(\cupl_{k\in\N}D\ke\right)=\Omega\setminus \Gamma\e.$$

	Next we introduce the Neumann Laplacian $\mathcal{A}\e$ in $\Omega\e$. Recall that we use the notation $\HS$ for both spaces $\L(\Omega\e)$ and $\L(\Omega)$, which coincide since the set $\Gamma\e$ has  measure zero.
	We  define the sesquilinear form $\a\e$ in $\HS$  by
	\begin{equation}
		\label{ae}
		\a\e[u,v]=
		(\nabla u,\nabla v)_{\L(\Omega\e)}=
		\ds\int_{\Omega\e}
		\suml_{i=1}^n {\partial u\over\partial x^i} {\partial \overline{ v}\over\partial x^i} \d x,\quad
		\dom(\a\e)= \H^1(\Omega\e).
	\end{equation} 
	This form 
	is symmetric, densely defined, closed, and positive. By the first representation theorem \cite[Chapter 6, Theorem 2.1]{Ka66}, there exists the unique self-adjoint and positive operator $\A\e$ associated with $\a\e$,
	i.e. $\dom(\A\e)\subset\dom(\a\e)$ and
	\begin{gather*}
		\forall u\in
		\dom(\A\e),\ \forall  v\in \dom(\a\e):\ (\A\e u,v)_{\HS}= \a\e[u,v].
	\end{gather*}
	Note that $u\e=(\A\e+\Id)f$ is the solution to the problem \eqref{BVPe}.

	Our goal is to describe the behaviour of the resolvent of $\A\e$  as $\eps\to 0$
	under the assumptions  on the distribution and  sizes of the holes $D\ke$ stated below (see \eqref{assump:1}--\eqref{assump:3} and \eqref{assump:main}).

	Let $d\ke$ and $x\ke$ be the radius and the center of the smallest ball $\B( d\ke,x\ke)$ containing $D\ke$; evidently,  $x\ke\in \Gamma$. 
	We make the following assumptions:
	there exist a sequence $\left\{\rho\ke,\ k\in \N\right\}$ 
	of positive numbers
	such that
	\begin{gather}
		\label{assump:1}B\ke\cap B_{l,\eps} =\emptyset\text{ for } k\not=l,\text{ where }B\ke\ceq \B(\rho\ke,x\ke),\\[1mm]
		\label{assump:2}\sup_{k\in\N}\rho\ke\to 0\text{ as }\eps\to0,\\	
		\label{assump:3}\sup_{\eps\in (0,\eps_0]}\sup_{k\in\N}\gamma\ke<\infty,
	\end{gather}	
	where   the numbers $\gamma\ke$ are defined by $$\gamma\ke\ceq 
	\begin{cases}
		d\ke^{n-2}{\rho}\ke^{1-n},&n\ge 3,\\
		|\ln d\ke|^{-1}\rho\ke^{-1},&n=2.
	\end{cases}$$
	It follows easily from \eqref{assump:2} and \eqref{assump:3} that 
	\begin{gather}\label{eps0:1}
		\sup_{\eps\in (0,\eps_0]}\sup_{k\in\N}\rho\ke<1,
		\\
		\label{eps0:1+}
		\sup_{\eps\in (0,\eps_0]}\sup_{k\in\N}\rho\ke< {\dist(\Gamma,\partial\Omega)\over 2},
		\\
		\label{eps0:2}
		\forall\eps\in (0,\eps_0],\ \forall k\in\N:\quad d\ke\le {\rho\ke\over 8}
	\end{gather}
	for  small enough $\eps_0$; in particular, \eqref{assump:1} and \eqref{eps0:2}    imply $\overline{D\ke}\cap \overline{D_{l,\eps}} =\emptyset$ for $k\not=l$. In the following we assume that $\eps_0$ is small enough so that the   conditions \eqref{eps0:1}--\eqref{eps0:2} are satisfied.
	
	To formulate our last assumption, we introduce the capacity-type   quantity 
	$\mathscr{C}(D\ke)$ by
	\begin{gather}\label{capty:inf} 
		\capty(D\ke)=\mathrm{inf}_{U\in \H^1_0(B\ke):\, U\restr_{\overline{D\ke}}=1}\|\nabla U\|^2_{\L(B\ke)}.
	\end{gather}
	Note that $\capty(D\ke)>0$ despite  $D\ke$ has Lebesgue measure zero.
	We   also introduce the set 
	\begin{gather}\label{Ske}
		S\ke\ceq B\ke\cap\Gamma, 
	\end{gather}
	and denote  
	by $\la v\ra_{S\ke}$  the mean value of a function $v:\Gamma\to\C$ in $S\ke$, i.e.
	\begin{equation*}
		\la v\ra_{S\ke}\ceq {1\over\area(S\ke)}\int_{S\ke}v(x)\d x',\quad x=(x',0)\in S\ke.
	\end{equation*}
	Then our last assumption reads as follows: there exists $\gamma \in C^{1}(\Gamma)\cap\mathsf{W}^{1,\infty}(\Gamma)$
	such that 
	\begin{gather}
		\label{assump:main}  
		\left|{1\over 4}\suml_{k\in\N}\capty(D\ke)
		\la g \ra_{S\ke}\la \overline{h} \ra_{S\ke}-(\gamma g,h)_{\L(\Gamma)}\right| \leq
		\kappa\e\| g\|_{\H^{3/2}(\Gamma)}\|h\|_{\H^{1/2}(\Gamma)}\text{\quad with\quad } \kappa\e\underset{\eps\to0}\to 0
	\end{gather}
	for any 
	$g\in \H^{3/2}(\Gamma)$ and
	$h\in \H^{1/2}(\Gamma)$.
	In Section~\ref{sec4} we present two examples on the distribution and the sizes of the holes for which the condition  \eqref{assump:main} is fulfilled.

	Now, we introduce the limiting operator $\mathcal{A}$.
	In the following, for $f\in\HS$, we denote by $f^+$ and $f^-$
	the restrictions of $f$ on $\Omega^+$ and $\Omega^-$, respectively.
	The same notations $f^\pm$ will be used for the traces of $f^\pm$ on $\Gamma$
	provided $f\in\H^1(\Omega\setminus\Gamma)$ (in this case the traces $f^\pm$ are indeed well-defined and belong to $\H^{1/2}(\Gamma)$).
	In  $\HS$ we define the sesquilinear form $\a$ via
	\begin{align} \notag
		\a[f,g]=(\nabla f^+, \nabla g^+)_{\L(\Omega^+)}&+(\nabla f^-, \nabla g^-)_{\L(\Omega^-)}\\
		&+(\gamma [f] ,[g])_{\L(\Gamma)},
		\quad
		\dom(\a)=\H^1(\Omega\setminus\Gamma),\label{a}
	\end{align}
	where $[\cdot] $ stands for the jump of the enclosed function  across $\Gamma$:
	$$
	[h] \ceq (h^+-h^-)\restriction_{\Gamma},\quad h\in\H^1(\Omega\setminus\Gamma).
	$$
	This form is symmetric, densely defined, closed, and non-negative \cite[Proposition 3.1]{BEL14}). We denote by $\A$ the associated self-adjoint operator.
	One has  \cite[Theorem~3.3]{BEL14}: 
	\begin{gather}\label{A:domain} 
		f\in\dom(\A)\ \Leftrightarrow\
		\begin{cases}
			f^\pm\in\H^1(\Omega^\pm),\\[1ex]
			\Delta f^\pm\in \L(\Omega^\pm),\\[1ex]\ds
			{\partial  f^\pm\over\partial   \nu^\pm}\in \L(\partial\Omega^\pm),\quad
			{\partial  f^\pm\over\partial   \nu^\pm}\restriction_\Gamma= 
			\mp\gamma[f] ,\quad
			{\partial  f^\pm\over\partial   \nu^\pm}\restriction_{\partial\Omega^\pm\setminus\Gamma}= 
			0.
		\end{cases}
	\end{gather}
	where  ${\partial   \over\partial   \nu^\pm}$ is the derivative along the outward (with respect to $\Omega^\pm$) pointing normal to $\partial\Omega^\pm$.
	The operator $\A$ acts as follows:
	\begin{gather*} 
		(\A f)^\pm = -\Delta f^\pm.
	\end{gather*}
	In fact, the functions from $\dom(\A)$ possesses even more regularity, but we postpone
	the corresponding statement until the proof of main results, see Lemma~\ref{lemma:H2}.

	\begin{remark}
		It is not hard to show that 
		if \eqref{assump:1}--\eqref{assump:3} hold and, moreover, the left-hand-side 
		of \eqref{assump:main} converge to zero for any fixed $g,h\in C^\infty_0(\Gamma)$, then
		\begin{gather}\label{prop:conv}
			{1\over 4}\ds\suml_{k\in\N}\capty(D\ke)\delta_{x\ke} \rightharpoonup
			\gamma\text{ in }\mathscr{D}'(\Gamma)\text{ as }\eps\to 0,
		\end{gather}
		where $\delta_{x\ke}$ is the Dirac delta function supported at $x\ke$.
		Furthermore, if the domain $\Omega$ is bounded, one can prove that $\A\e$ converges to $\A$ is the strong resolvent sense (i.e., \eqref{L2strong} holds) provided  the assumptions \eqref{assump:1}--\eqref{assump:3} and \eqref{prop:conv} are satisfied. However, to derive operator estimates we require stronger assumption \eqref{assump:main}
		instead of \eqref{prop:conv}.
	\end{remark}

	We are  now in position to formulate the   main results of this work.
	For $k\in\N$ we set 
	\begin{gather}
		\label{etake}
		\eta\ke\ceq  
		\begin{cases}
			{d\ke\over \rho\ke} ,&n\ge 5,
			\\[2mm] {d\ke\over \rho\ke} \left|\ln {d\ke\over \rho\ke} \right|,&n=4,
			\\[2mm] 
			\left({d\ke\over \rho\ke}\right)^{1/2},&n=3,
			\\[2mm] 
			\left|\ln {d\ke\over \rho\ke} \right|^{-1/2},&n=2.
		\end{cases}
	\end{gather}    
	It follows from \eqref{assump:2}, \eqref{assump:3} that $\sup_{k\in\N}\eta\ke\to 0$ as $\eps\to 0$.
	
	\begin{theorem}\label{th1}
		One has
		\begin{gather}\label{th1:est}
			\forall f\in\HS:\quad \|(\A\e+\Id)^{-1}f -(\A+\Id)^{-1}f  \|_{\HS}\leq 
			C\mu\e\|f\|_{\HS},
		\end{gather}
		where $\mu\e$ is defined by 
		\begin{gather}\label{etae}
			\mu\e=	
			\begin{cases}
				\max\left\{\kappa\e;\,  \supl_{k\in\N}\eta\ke \right\},&n\ge 3,\\[3mm]
				\max\left\{\kappa\e;\,  \supl_{k\in\N}\eta\ke; \,  \supl_{k\in\N}\left(\rho\ke^{1/2}\gamma\ke|\ln\rho\ke|\right) \right\},&n=2.
			\end{cases}
		\end{gather}

	\end{theorem}	
	
	An immediate consequence of Theorem~\ref{th1} is a convergence of spectra. Recall that
	for closed sets $X,Y\subset\R$  the \emph{Hausdorff distance} $d_H (X,Y)$ is given by
	\begin{gather*}
		d_H (X,Y)\ceq\max\left\{\sup_{x\in X} \inf_{y\in Y}|x-y|;\,\sup_{y\in Y} \inf_{x\in X}|y-x|\right\}.
	\end{gather*}
	
	\begin{corollary}\label{th3}
		One has:
		\begin{itemize}
			\item[(A)] For any $\lambda\in \sigma(\A)$ there exists a family $(\lambda\e)_{\eps>0}$ with 
			$\lambda\e\in \sigma(\A\e)$ such that $\lambda\e\to \lambda$ as $\eps\to 0$;
			
			\item[(B)] For any $\lambda\in\R\setminus \sigma(\A)$ there exist $\delta>0$ 
			such that $\sigma(\A\e)\cap (\lambda-\delta,\lambda+\delta)=\emptyset$ 
			for sufficiently small $\eps$.
			
		\end{itemize}
		Moreover, one has the estimate
		\begin{gather}\label{th3:est}
			d_{H}(\sigma((\A\e+\Id)^{-1}),\sigma((\A+\Id)^{-1})\leq C\mu\e.
		\end{gather}
	\end{corollary}
	
	\begin{proof}
		For two normal bounded operators $R_1$ and $R_2$  in the Hilbert space $\HS$.
		one has the following inequality (see, e.g., \cite[Lemma~A.1]{HN99}):
		$$
		d_H(\sigma(R_1),\sigma(R_2))\leq \|R_1-R_2\|_{\HS}.
		$$
		Applying it for $R_1\ceq (\A\e+\Id)^{-1}$ and  $R_2\ceq (\A+\Id)^{-1}$ and
		taking into account \eqref{th1:est}, we arrive at the estimate \eqref{th3:est}.
		By virtue of 
		\cite[Lemma~A.2]{HN99} the statements $(A)$ and $(B)$ follow immediately from the fact 
		$d_{H}(\sigma((\A\e+\Id)^{-1}),\sigma((\A+\Id)^{-1})\to 0$ as $\eps\to 0$.
	\end{proof}
	
	The second main theorem establishes the closeness of  resolvents 
	$(\A\e+\Id)^{-1}$ and $(\A+\Id)^{-1}$ in the $(\L\to \H^1)$ operator norm.
	Here, to get a reasonable result, one needs to employ a special
	corrector. We denote
	\begin{gather}\label{Bke}
		B\ke^\pm\ceq B\ke\cap\Omega^\pm.
	\end{gather}
	Let $U\ke\in H^1_0(B\ke)$ be the solution to the problem
	\begin{gather}\label{BVP:cap}
		\begin{cases}
			\Delta U=0,&x\in B\ke\setminus\overline{D\ke},\\
			U=1,&x\in\partial D\ke =\overline{D\ke},\\
			U=0,&x\in\partial B\ke.
		\end{cases}    
	\end{gather}  
	It is well-known that
	\begin{gather}\label{capty:eq}
		\capty(D\ke)=
		\|\nabla U\ke\|^2_{\L(B\ke)}.
	\end{gather}
	We introduce the operator $\mathscr{K}\e:\HS\to \H^1(\Omega\setminus\Gamma)$ via
	\begin{gather}\label{Ke}
		(\mathscr{K}\e f)(x) =
		\begin{cases}
			-{1\over 2} \la [g]\ra_{S\ke}U\ke^+(x),&x\in B\ke^+ ,\\
			{1\over 2} \la [g]\ra_{S\ke}U\ke^-(x),&x\in B\ke^- ,\\
			0,&x\in\ds\Omega\setminus\left(\cup_{k\in\N} B\ke\right),
		\end{cases}\quad\text{where } 
		g= (\A+\Id)^{-1}f. 
	\end{gather}
	Here  $U\ke^\pm\ceq U\ke\restriction_{B\ke^\pm}$; recall that
	$[g]=(g^+-g^-)\restriction_\Gamma$, and $\la [g]\ra_{S\ke}$ stands for the mean value of $[g]$ in $S\ke$.

	\begin{theorem}\label{th2}
		One has 
		\begin{gather}\label{th2:est}
			\forall f\in\HS:\quad \|(\A\e+\Id)^{-1}f - (\A+\Id)^{-1}f - \mathscr{K}\e f \|_{\H^1(\Omega\setminus\Gamma)} \leq C\mu\e\|f\|_{\HS}.
		\end{gather}
		The correcting term $\mathscr{K}\e f$ obeys the following   properties: 
		\begin{gather}\label{K:prop1}
			\|(\mathscr{K}\e f)^\pm\|_{\L(\Omega^\pm)}\leq  C\sup_{k\in\N}(\eta\ke\rho\ke^{1/2})\|f\|_{\HS},\\\label{K:prop2}
			\|\nabla (\mathscr{K}\e f)^\pm\|^2_{\L(\Omega^\pm)}=
			{1\over 2}\|\ga^{1/2}[g]\|^2_{\L(\Gamma)}+\mathcal{O}(\kappa\e)\|f\|^2_{\HS},\text{ where }
			g=(\A+\Id)^{-1}f.
		\end{gather}
	\end{theorem}
	
	\begin{remark}\label{rem:Gamma'} 
	Besides $\Omega$   satisfying \eqref{h} one can also address the case when $\Gamma$ intersects $\Omega$. However, here we faced technical difficulties that so far we cannot overcome -- 
	see Remark~\ref{rem:intersection} for    details. The only case, which we still are able to handle
	is when the holes are located on a ``safe'' distance from $\partial\Omega$. Namely,  let $\Omega\subset\R^n$ be a Lipschitz domain in $\R^n$, while the hyperplane $\Gamma$  
	intersects $\Omega$, and there is a closed subset
	$\Gamma'$ of $\Gamma\cap\Omega$ such that 
	$\dist(\Gamma',\partial\Omega)>0$. We assume that $\{D\ke\subset\Gamma',\ k\in\mathbb{M}\e\subseteq\N\}$ are relatively open connected sets satisfying the assumptions \eqref{assump:1}--\eqref{assump:3}  (with $k\in\N$ being replaced by $k\in\mathbb{M}\e$), and there is a function $\gamma \in C^{1}(\Gamma)\cap\mathsf{W}^{1,\infty}(\Gamma)$ with $\supp(\gamma)\subset\Gamma'$ such that the assumption
	\eqref{assump:main} is fulfilled. 
	In this case Theorems~\ref{th1} and \ref{th2} remain 
	valid with the homogenized operator $\A$ being defined in a similar way. The assumption 
	$\supp(\gamma)\subset\Gamma'$ means that we get non-trivial interface conditions only 
	on $\Gamma'$, while the set $(\Gamma\cap\Omega)\setminus\Gamma'$ is 
	impenetrable -- from both sizes we have the Neumann conditions. 
	The proof is similar to the case \eqref{h}.
	\end{remark}

	\begin{remark}
		Instead of Laplacians one can also consider Schr\"odinger operators. Namely, 
		let $V\in\LL^\infty(\Omega)$ be real, then the estimates
		\eqref{th1:est}, \eqref{th2:est}   hold true with
		$\A\e+\Id$ and $\A+\Id$ being replaced by
		$\A\e+V-\Lambda\Id$ and $\A+V-\Lambda\Id$, respectively;   $\Lambda\in\R$ is  small enough
		in order to be in the resolvent sets of $\A\e+V$ and $\A+V$
		(for example, one can choose
		$\Lambda=-\|V\|_{\LL^\infty(\Omega)}-1$).
		\end{remark}

	\begin{remark}
	The choice of the boundary conditions on $\partial\Omega$
	is inessential: the above results remain valid if we imposes Dirichlet, Robin, mixed or any other \emph{$\eps$-independent}   conditions on $\partial\Omega$ (of course, the homogenized operator then must be changed accordingly). 	
    \end{remark}

	The remaining part of the work is organized as follows.
	In the next section we present the proof of Theorems~\ref{th1} and \ref{th2}.
	In Subsection~\ref{subsec:3:1} we collect several useful estimates which will be widely used further. In Subsection~\ref{subsec:3:2} 
	we recall the abstract result from \cite{AP21} for studying convergence of operators in varying Hilbert spaces. In Subsection~\ref{subsec:3:3} we verify the conditions of this abstract result in our concrete setting. The proofs of  are completed in
	Subsection~\ref{subsec:3:4}.  
	Finally, in Section~\ref{sec4} we present two examples  for which the condition    \eqref{assump:main} is fulfilled.

	\section{Proof of main results}\label{sec3}

	\subsection{Auxiliary estimates}\label{subsec:3:1}

	In the following, by $\la v\ra _{B}$ we denote the mean value of the function  $v$ in the   
	measurable bounded set $B$ with $\vol(B)\not=0$ (hereinafter, $\vol(B)$ stands for the volume of $B$), i.e.,
	\begin{equation*}
		\la v\ra_{B}\ceq {1\over \vol(B)}\int_{B}v(x)\d x.
	\end{equation*}
	The same notation will be use for the mean value of a function defined on a subset $S$ of the hyperplane $\Gamma$ with $\area(S)\not=0$: 
	\begin{equation*}
		\la v\ra_{S}\ceq {1\over\area(S)}\int_{S}v(x)\d x',\ x=(x',0).
	\end{equation*}
	
	Recall that $S\ke^\pm$ and $B\ke^\pm$ are defined by \eqref{Ske} and \eqref{Bke}, respectively.

	\begin{lemma}
		\label{lemma:poincare:friedrichs:trace}
		One has 
		\begin{align} \label{Poincare}
			\forall v\in \H^1(B\ke^\pm):&&
			\|v-\langle v \rangle_{B\ke^\pm}\|^2_{\L(B\ke^\pm)} \leq
			C\rho\ke^2 \|\nabla v\|^2_{\L(B\ke^\pm)},\\ \label{Friedrichs} 
			\forall v\in \H^1_0(B\ke):&& 
			\|v \|^2_{\L(B\ke )} \leq
			C\rho\ke^2 \|\nabla v\|^2_{\L(B\ke)},\\ \label{trace}
			\forall v\in \H^1(B\ke^\pm):&& 
			\|v \|^2_{\L(B\ke^\pm)} \leq
			C\left(\rho\ke\|v \|_{\L(S\ke)}^2+\rho\ke^2\|\nabla v\|^2_{\L(B\ke^\pm)}\right).
			\\ \label{mean:mean}
			\forall v\in \H^1(B\ke^\pm):&& 
			|\la v\ra_{S\ke }-\la v\ra_{B\ke^\pm} |^2  \leq
			C \rho\ke ^{2-n}\|\nabla v \|^2_{\L(B\ke^\pm)}.
		\end{align}
	\end{lemma}
	
	\begin{proof}We denote
		$B\ceq \B(1,0)$, $B^\pm\ceq B\cap \left\{x\in\R^n:\ \pm x^n>0\right\}$, $S\ceq B\cap\Gamma$.
		We have the following standard Poincar\'e-type inequalities:
		\begin{align}\label{Poincare:1}
			\forall v\in \H^1(B^\pm):&&
			\|v-\langle v \rangle_{B^\pm}\|^2_{\L(B^\pm)} \leq
			\Lambda_N^{-1}\|\nabla v\|^2_{\L(B^\pm)},\\
			\label{Friedrichs:1}
			\forall v\in \H^1_0(B):&&
			\|v\|_{\L(B )}^2 \leq
			\Lambda_D^{-1 }\|\nabla v\|_{\L(B)}^2,\\\label{trace:1}
			\forall v\in \H^1(B^\pm):&&	\|v\|_{\L(B^\pm)}^2\leq 
			\Lambda_R^{-1}
			\left(\|v\|^2_{\L( S)}+\|\nabla v\|^2_{\L(B^\pm)}\right),
		\end{align}
		where $\Lambda_N>0$  
		is the   smallest non-zero eigenvalue of 
		the Neumann Laplacian on $B^\pm$, 
		$\Lambda_D>0$ 
		is the   smallest  eigenvalue of 
		the  Dirichlet Laplacian on $B $, and
		$\Lambda_R>0$
		is  the   smallest  eigenvalue of 
		the   Laplacian on $B^\pm$ subject to the Robin boundary conditions 
		${\partial  v\over\partial   \nu^\pm} +v=0$ ($\nu^\pm$ is the  unit normal  pointed outward of $B^\pm$) on $S$ and the Neumann boundary conditions on
		$\partial B^\pm\setminus S$.
		Then, via the coordinate transformation
		\begin{gather}\label{variables}
			y=  \rho\ke^{-1}(x-x\ke),
		\end{gather}
		we reduce  \eqref{Poincare:1}--\eqref{trace:1} to \eqref{Poincare}--\eqref{trace}.
		Furthermore, one has also the trace estimate
		\begin{gather*}
			\forall v\in \H^1(B^\pm):\quad	\|v\|_{\L(S)}^2\leq 
			C\| v\|^2_{\H^1(B^\pm)},
		\end{gather*}
		which, after the coordinate transformation \eqref{variables}, reduces to
		\begin{align} \label{trace:3}
			\forall v\in \H^1(B\ke^\pm):\quad
			\|v \|^2_{\L(S\ke)} \leq
			C\left(\rho\ke^{-1}\|v \|_{\L(B\ke^\pm)}^2+\rho\ke\|\nabla v\|^2_{\L(B\ke^\pm)}\right).
		\end{align}
		Using the Cauchy-Schwarz inequality,  \eqref{Poincare} and  \eqref{trace:3}, we get
		the last estimate \eqref{mean:mean}:
		\begin{align*}
			\left|\la v\ra_{S\ke}-\la v\ra_{B\ke^\pm} \right|^2 &=
			\left|\la v-\la v\ra_{B\ke^\pm}\ra_{S\ke } \right|^2\leq
			{1\over \area(S\ke)}\| v-\la v\ra_{B\ke^\pm}\|^2_{\L(S\ke)}\\
			&\leq 
			C\rho\ke^{1-n}\left(\rho\ke^{-1}\|v- \la v\ra_{B\ke^\pm}\|^2_{\L(B^\pm\ke)}+\rho\ke\|\nabla v\|^2_{\L(B^\pm\ke)} \right)\leq
			C_1\rho\ke^{2-n}\|\nabla v\|^2_{\L(B\ke)}.
		\end{align*}
		The lemma is proven.
	\end{proof}

	We denote
	\begin{gather} \label{wtdke}
		\wt d\ke\ceq
		\begin{cases}
			2 d\ke ,&n\ge 3 , 
			\\
			(\rho\ke d\ke)^{1/2} ,&n=2, 
		\end{cases}\\\label{Fke}
		\wt D\ke^\pm\ceq \B(\wt d\ke,x\ke)\cap\left\{x\in\R^n:\ \pm x^n>0\right\}.
	\end{gather}
	It follows \eqref{eps0:2} that
	\begin{gather}\label{wtd:cond}
		\forall\eps\in (0,\eps_0],\ \forall k\in\N:\quad \wt d\ke<{\rho\ke\over 2}.
	\end{gather}
	whence, in particular, $\overline{\wt D\ke^\pm}\subset B\ke^\pm\cup S\ke$.

	\begin{lemma}\label{lemma:Fest}
		One has  
		\begin{gather}\label{Fest}
			\forall v\in \H^1(B^\pm\e):\
			\|v\|^2_{\L(\wt D\ke^\pm)}\leq
			\begin{cases}
				C\left(\left({d\ke \over \rho\ke}\right)^n\|v\|^2_{\L(B\ke^\pm)}+d\ke^2\|\nabla v\|^2_{\L(B\ke^\pm)}\right),&n\ge 3,\\[3mm]
				C\left({d\ke \over \rho\ke}\|v\|^2_{\L(B\ke^\pm)}+d\ke \rho\ke|\ln (d\ke\rho\ke)|\cdot\|\nabla v\|^2_{\L(B\ke^\pm)}\right),&n=2.
			\end{cases}
		\end{gather}
	\end{lemma}
	
	\begin{proof}
		Evidently, it is enough to prove \eqref{Fest} for $v\in  {C}^\infty(\overline{B^\pm\e})$.
		
		We denote $C\ke^\pm\ceq \partial \wt D\ke^\pm\setminus\Gamma$.
		Similarly to \eqref{trace}, we obtain the estimate 
		\begin{align}\label{lm:Fest:est1}
			\|v\|_{\L(\wt D\ke^\pm)}^2\leq 
			C\left(\wt d\ke\|v\|^2_{\L(C\ke^\pm)}+
			(\wt d\ke)^2\|\nabla v\|^2_{\L(\wt D\ke^\pm)}\right). 
		\end{align}
		We introduce  the spherical coordinate $(r,\theta)$ in $\overline{B\e^\pm\setminus \wt D\ke^\pm}$.
		Here $r\in [\wt d\ke,\rho\ke]$ stands for the distance to $x\ke$,
		$\theta=(\theta_1,{\dots},\theta_{n-1})$ are the angular coordinates ($\theta_k\in [0,\pi]$).
		One has
		$$
		v(\wt d\ke,\theta)=v(r,\theta)-\int_{\wt d\ke}^{r}{\partial
			v(\tau,\theta)\over\partial\tau}\d\tau,
		$$
		whence
		\begin{align*}
			|v(\wt d\ke,\theta)|^2&\le 2|v(r,\theta)|^2+2\left|\int_{\wt d\ke}^{r}{\partial
				v(\tau,\theta)\over\partial\tau}\d\tau\right|^2\\
			\notag
			&\leq
			2|v(r,\theta)|^2+
			2M\e \int_{\wt d\ke}^{\rho\ke}\left|{\partial
				v(\tau,\theta)\over\partial\tau}\right|^2\tau^{n-1}\d\tau,\text{\quad where }
			M\e\ceq \ds\int_{\wt d\ke}^{\rho\ke}\tau^{1-n}\d\tau.
		\end{align*}
		We denote $N\e\ceq \ds \int_{\wt d\ke}^{\rho\ke} r^{n-1}\d r $.
		Multiplying the estimate above by $$(N\e)^{-1}r^{n-1}(\wt d\ke)^{n-1}\prod_{j=1}^{n-2}\left(\sin\theta_j\right)^{n-1-j},$$ and then integrating
		over $r\in (\wt d\ke,\rho\ke)$, $\theta_j\in (0,\pi)$, $j=1,\dots,n-1$,
		we get
		\begin{equation}\label{lm:Fest:est3}
			\|v\|^2_{\L(C\ke^\pm)}
			\leq 2(\wt d\ke)^{n-1}\left((N\e)^{-1}\|v\|^2_{\L(B^\pm\ke\setminus \overline{\wt D^\pm\ke})}
			+
			M\e\|\nabla v\|^2_{\L(B^\pm\ke\setminus \overline{\wt D\ke^\pm})}\right).
		\end{equation}
		It is easy to see that
		\begin{gather}\label{M:eps}
			M\e\leq \begin{cases}C(\wt d\ke)^{ 2-n},&n\ge 3,\\
				C|\ln \wt d\ke|,&n=2
			\end{cases}
		\end{gather}
		(in the case $n=2$ we took into account that, due to \eqref{eps0:1}, \eqref{wtd:cond}, $\ln \wt d\ke<\ln \rho\ke <0$).
		Furthermore, due to  \eqref{wtd:cond}, one has
		\begin{gather}\label{N:eps}
			N\e\geq C\rho\ke^n.
		\end{gather}
		The required estimate \eqref{Fest} follows from
		\eqref{wtdke}, \eqref{lm:Fest:est1}--\eqref{N:eps} 
		(in the case $n=2$ we also have  take into account that, by virtue of \eqref{eps0:1}, \eqref{wtd:cond}, 
		$|\ln \wt d\ke|\geq C>0$). The lemma is proven.
	\end{proof}

	Let  $\square\subset\R^n$ be a cube, and  $\square\e\cong\eps\square$.
	One has the following estimate  \cite[Lemma~4.3]{KP18}:
	\begin{gather*}
		\forall v\in \mathsf{H}^2(\square\e)\text{ with }\int_{\square\e}v(x)\d x=0:\quad
		\|v\|_{\LL^p(\square\e)}\leq  C_{n,p}\cdot \eps^{n/p+(2-n)/2}\|v\|_{\mathsf{H}^2(\square\e)},
	\end{gather*}
	provided $p$ satisfies 
	\begin{gather}
		\label{p}
		1\leq p\leq \frac{2n}{n-4}\text{\; as\; }n\geq 5,\quad
		1\leq p<\infty\text{\; as\; }n=4,\quad
		1\le p\le \infty\text{\; as\; }n=2,3
	\end{gather} 
	(hereinafter for $p=\infty$ we use the convention $1/p=0$). Note that 
	the above restrictions on $p$  came from the Sobolev embedding theorem  \cite[Theorem~5.4 and Remark~5.5(6)]{Ad75}:
	if $D\subset\R^n$ is a bounded Lipschitz domain, then the space $\H^2(D)$ is embedded continuously into the space
	$\LL^{p}(D)$ provided $p$ satisfies \eqref{p}.
	Repeating verbatim the proof of \cite[Lemma~4.3]{KP18} for $B\ke^\pm\cong\rho\ke B^\pm$ 
	one arrives on the estimate below.
	
	\begin{lemma} 
		\label{lemma:sobolev} 
		One has
		\begin{gather}\label{sobolev}
			\forall v\in \mathsf{H}^2(B\ke^\pm)\text{ with }\int_{B\ke^\pm}v(x)\d x=0:\ 
			\|v\|_{\LL^p(B\ke^\pm)}\leq  C_{n,p}\cdot \rho\ke^{n/p+(2-n)/2}\|v\|_{\mathsf{H}^2(B\ke^\pm)}
		\end{gather}
		provided $p$  satisfies \eqref{p}. The constant $C_{n,p}$ depends on $n$ and $p$ only.
	\end{lemma}
	
	\begin{remark}
		Further  we will apply \eqref{sobolev} for the \emph{largest} $p$ satisfying \eqref{p}.
		For $n=4$ we are not able to choose the largest $p$ (in the dimension $4$ the embedding $\H^2\hookrightarrow\LL^p$ holds for any $p<\infty$, but not for $p=\infty$), and in this case we need the following estimate on the constant $C_{4,p}$ in the right-hand-side of \eqref{sobolev}
		\cite[Lemma~4.4]{KP21}:
		\begin{gather}\label{lemma:c4p:est}
			C_{4,p}\leq C  p,
		\end{gather}
		where the constant $C>0$ is independent of $p$. 
	\end{remark}

	\subsection{Abstract scheme}\label{subsec:3:2}
	
	Let $\HS$ be a Hilbert space, $\A\e$ and $\A$ be non-negative, self-adjoint, unbounded operators in $\HS$,  
	and 
	$\a\e$ and $\a$ be the associated 
	sesquilinear forms. Further, we introduce
	the energetic spaces $\HS^1\e$ and $\HS^1$ associated with the forms $\a\e$ and $\a$, respectively:
	\begin{equation}\label{H1}
		\begin{array}{ll}
			\HS^1\e=\dom(\a\e),&
			\|u\|_{\HS^1\e}^2=\a\e[u,u]+\|u\|^2_{\HS},\\[2mm]
			\HS^1=\dom(\a),&
			\|f\|_{\HS^1}^2=\a[f,f]+\|f\|^2_{\HS}.
		\end{array}
	\end{equation}
	We also introduce the Hilbert space   $\HS^2$ via
	\begin{gather} \label{H2}
		\HS^2\ceq\dom(\A),\quad
		\|f\|_{\HS^2}\ceq\|{(\A+\Id) f}\|_{\HS}.
	\end{gather} 
	
	The proof of Theorems~\ref{th1} and \ref{th2} is based on the following  abstract result from \cite{AP21}.

	\begin{proposition}[{\cite[Proposition~2.5]{AP21}}]
		\label{propA1}
		Let 
		$\J\e \colon \HS^1\to  \HS\e^1$, ${\J'\e }\colon {\HS^1\e}\to \HS^1$
		be linear operators satisfying 
		\begin{align} 
			\label{propA1:1}
			\|\J\e f- f\|_{\HS\e}&\leq \delta\e\|f\|_{\HS^1 },&& \forall f\in \HS^1 ,
			\\[1mm] 
			\label{propA1:2}
			\|\J'\e u -  u \|_{\HS}&\leq 
			\delta\e\|u\|_{ \HS\e^1},&&  \forall u\in \HS^1\e, 
			\\[1mm]
			\label{propA1:3}
			|\a\e[\J\e f,u]-\a[f,\J'\e u]  |&\leq 
			\delta\e\|f\|_{\HS^2 }\|u\|_{\HS^1\e},&& \forall f\in \HS^2 ,\ u\in \HS^1\e 
		\end{align}
		with some  $\delta\e\geq 0$.
		Then one has the estimate 
		\begin{align*}
			\forall f\in\HS:\quad	\|(\A\e+\Id)^{-1}f -\J\e (\A+\Id)^{-1}f\|_{\HS^1\e}\leq C\delta\e\|f\|_{\HS},
		\end{align*}
	with some absolute constant $C>0$.
	\end{proposition}

	\begin{remark}
		In fact, Proposition~2.5 from \cite{AP21} covers even more general case, when the operators
		$\A\e$ and $\A$ act in distinct Hilbert spaces $\HS\e$ and $\HS$;
		in this  case one requires also certain identification operators between $\HS\e$ and $\HS$.
		However, if these spaces coincide and these  identification operators are chosen to be identity, Proposition~2.5 from \cite{AP21} 		reduces to Proposition~\ref{propA1}. 
	\end{remark}

	\subsection{Utilization of the abstract scheme}\label{subsec:3:3}
	
	We apply Proposition~\ref{propA1}   for 
	$
	\HS=\L(\Omega)=\L(\Omega\e)$,
	and the self-adjoint non-negative operators $\A\e$ and $\A$ being associated with the sesquilinear forms $\a\e$ \eqref{ae}  and $\a$  \eqref{a}, respectively. 
	We introduce the spaces $\HS\e^1$, $\HS^1$, $\HS^2$ as in \eqref{H1} and \eqref{H2}, i.e.
	\begin{gather*} 
		\begin{array}{lll}
			\HS^1\e=\H^1(\Omega\e),
			&
			\ds\|u\|^2_{\HS^1\e}
			&=\|u \|^2_{\H^1(\Omega\e)},
			\\[1.5ex]
			\HS^1=\H^1(\Omega\setminus\Gamma),
			&
			\ds\|f\|^2_{\HS^1}&=
			\|\nabla f^+\|^2_{\L(\Omega^+)}+
			\|\nabla f^-\|^2_{\L(\Omega^-)} 
			+\|\gamma^{1/2}[f]\|^2_{\L(\Gamma)}+\|f \|^2_{\HS},
			\\[1.5ex]
			\HS^2=\dom(\A),
			&\ds
			\|f\|_{\HS^2}^2
			&
			=\|-\Delta f^+ +f^+ \|^2_{\L(\Omega^+)}+\|-\Delta f^- +f^- \|^2_{\L(\Omega^-)}
		\end{array}
	\end{gather*} 
	(recall that   $f^\pm=f\restriction_{\Omega^\pm}$, and  $[f]=(f^+-f^-)\restriction_{\Gamma}$). Note that 
	\begin{gather}\label{HHH}
		\|f\|_{\HS}\leq \|f\|_{\HS^1}\leq \|f\|_{\HS^2}. 
	\end{gather}
	
	Since $\dom(\a\e)\subset\dom(\a)$, we define  
	$\J'\e: \HS\e^1\to \HS^1$ be equal to the identity operator:
	\begin{gather*}\label{wtJ}
		\J'\e u = u,\ u\in \H^1(\Omega\e).
	\end{gather*}
	With such a choice of   $\J'\e$ 
	\begin{gather}\label{automatic}
		\text{condition  \eqref{propA1:2} is fulfilled with $\delta\e=0$.}
	\end{gather}
	
	To define an appropriate operator $\J\e:\HS^1\to\HS\e^1$, we first need to introduce 
	two auxiliary functions $\phi\ke\in C(\R^n)$ and $\psi\ke\in C^\infty(\R^n)$
	via
	\begin{gather}\label{phipsi}
	\phi\ke(x)=
	\begin{cases}
		1,&|x-x\ke|\le d\ke,\\
		\ds
		{G(|x-x\ke|)- G(\wt d\ke) \over G(d\ke) - G(\wt d\ke)},&
		d\ke< |x-x\ke|< \wt d\ke,
		\\
		0,&\wt d\ke\le |x-x\ke|,
	\end{cases} \qquad \psi\ke(x)\ceq \psi\left({4|x-x\ke|\over \rho\ke}\right)
	\end{gather}
	Here $\wt d\ke$ is defined by \eqref{wtdke}, the function  $G:(0,\infty)\to\R$ is given by 
	\begin{gather}\label{Gt}
		G(t)\ceq
		\begin{cases}
			t^{2-n},&n\ge 3,\\
			-\ln t,&n=2,
		\end{cases}    
	\end{gather} 
	and $\psi\in {C}^\infty([0,\infty))$ is a fixed  cut-off  function  satisfying
	\begin{gather*}
		0\le \psi(t)\leq 1,\quad 
		\psi(t)=1\text{ as }t\le 1,\quad
		\psi(t)=0\text{ as }t\ge 2.
	\end{gather*}	
	
	Let $f\in\H^1(\Omega\setminus\Gamma)$. We define  
	\begin{gather*}
		(\J\e f)(x)\ceq  
		\begin{cases}
			\ds f^+ (x)+\suml_{k\in\N}\left(
			\left(\langle f^+ \rangle_{B\ke^+}-f^+(x)\right) \phi^+\ke(x)-
			{1\over 2}f\ke U^+\ke(x)\psi\ke^+(x)
			\right),& x\in\Omega^+,
			\\[2mm]
			\ds f^-(x)+\suml_{k\in\N}\left(
			\left(\langle f^- \rangle_{B\ke^-}-f^-(x)\right) \phi^-\ke(x)+
			{1\over 2}f\ke U^-\ke(x)\psi\ke^-(x)
			\right),& x\in\Omega^-.
		\end{cases}
	\end{gather*}
	Here $\phi\ke^\pm\ceq\phi\ke\restr_{\Omega^\pm}$, $\psi\ke^\pm\ceq\psi\ke\restr_{\Omega^\pm}$,
	$U\ke^\pm\ceq U\ke\restr_{\Omega^\pm}$, where $U\ke(x)$ is the solution to   \eqref{BVP:cap}  extended by $0$ to  $\R^n\setminus B\ke$, and 
	$$f\ke\ceq \la f^+ \ra_{B\ke^+}-\la f^- \ra_{B\ke^-}.$$
	One has $\supp(\phi\ke)\subset B\ke$,  
	$\phi\ke=1$ in
	$\overline{\B(d\ke,x\ke)}$; the same hold for  $\psi\ke$ (cf.~\eqref{eps0:2}). Using these properties we conclude that 
	$(\J\e f)^\pm\in \H^1(\Omega^\pm)$,
	furthermore,  
	the traces of $(\J\e f)^+$ and $(\J\e f)^-$
	coincide on $D\ke$ (with the constant  ${1\over 2}\big(\la f^+ \ra_{B\ke^+}+\la f^- \ra_{B\ke^-}\big)$).
	Thus, $\J\e f\in \H^1(\Omega\e)$, and we have a well-defined linear operator $\J\e:\HS^1\to\HS^1\e$. 
	
	Our goal is to show that the above introduced operators $\J\e$ and $\J'\e$ satisfy 
	the conditions \eqref{propA1:1} and \eqref{propA1:3} with some $\delta\e\to 0$.
	
	Firth, we establish some properties of the functions from 	the domain of the limiting operator $\A$.
	Recall that we assume $\dist(\Gamma,\partial\Omega)>0$ (cf.~\eqref{h}). We introduce the domains 
	\begin{gather*}
		O^\pm\ceq \left\{x=(x',x^n):\  |x^n|<{\dist(\Gamma,\partial\Omega)\over 2},\ \pm x^n>0\right\}.
	\end{gather*}
	
	\begin{lemma}\label{lemma:H2}
		Let $f\in\HS^2$. Then  $f^\pm \in \H^2(O^\pm)$. 
		Furthermore,  the following estimate holds true:
		\begin{gather}\label{H2est}
			\forall f\in\HS^2:\quad \|f^\pm \|_{\mathsf{H}^2(O^\pm)}\leq
			C \|f\| _{\HS^2} .
		\end{gather}
	\end{lemma}
	
	\begin{proof}
		For $i=(i^1,\dots,i^{n-1})\in\Z^{n-1}$ 
		we introduce the sets
		\begin{align*}
			\square_{i}& \ceq 
			\left\{x=(x^1,\dots,x^{n-1},0)\in\Gamma :\ |x^j - i^j|<{1\over 2} ,\ j=1,\dots,n-1 \right\},
			\\
			\wh\square_{i}& \ceq 
			\left\{x=(x^1,\dots,x^{n-1},0)\in\Gamma :\ |x^j - i^j|<1 ,\ j=1,\dots,n-1 \right\},
			\\
			{ O_{i}^\pm}&\ceq 
			\left\{x=
			(x',x^n)\in\R^n :\   x'\in\square_i,\ |x^n|<{1\over 2}\dist(\Gamma,\partial\Omega),\ \pm x^n>0\right\},
			\\
			\wh O_{i}^\pm&\ceq 
			\left\{x=(x',x^n)\in\R^n :\   x'\in\wh\square_i,\ |x^n|< \dist(\Gamma,\partial\Omega),\ \pm x^n>0\right\}.
		\end{align*}
		Note that 
		\begin{gather}\label{cup:square}
			\Gamma=\cupl_{i\in\Z^n}\overline{ \square_{i}},
			\quad
			\overline{O^\pm}=\cupl_{i\in\Z^n}\overline{ O^\pm_{i}},
			\quad 
			\square_{i}\cap \square_{j}=\emptyset,\ 
			O^\pm_{i}\cap  O^\pm_{j}=\emptyset
			\text{ if }i\not= j.
		\end{gather}
		One has (cf.~\eqref{A:domain}):
		\begin{gather}\label{regularies:McLean:1}
			f^\pm \in \H^1(\wh O_{i}^\pm),\quad  
			\Delta f^\pm \in\L(\wh O_{i}^\pm).
		\end{gather}
	 Since  $ f^\pm \in \H^{1/2}(\wh\square_{i})$ and $\gamma\in C^1(\Gamma)$,
			we have $\gamma f^\pm\in \H^{1/2}(\wh\square_{i})$ (see, e.g., \cite[Theorem~3.20]{McL200}),   whence
		\begin{gather}\label{regularies:McLean:2}
			{\partial  f^\pm\over\partial   \nu^\pm}\restriction_{\wh\square_{i}}=\mp\gamma[f]\restriction_{\wh\square_{i}}\in \H^{1/2}(\wh\square_{i}).
		\end{gather}
		By virtue of  \cite[Theorem~4.18(ii)]{McL200})  
		the properties  \eqref{regularies:McLean:1}--\eqref{regularies:McLean:2}  imply
		$f \in\H^2( O_{i}^\pm)$, moreover, the following estimate holds true:
		\begin{gather}\label{H2est:0}
			\|f^\pm \|^2_{\mathsf{H}^2(  O_{i}^\pm)}\leq
			C\left(\|f^\pm \|^2_{\H^1(\wh O_{i}^\pm)}+
			\|\gamma[f] \|_{\H^{1/2}(\wh\square_{i}^\pm)}^2+\|-\Delta f^\pm+f^\pm \|_{\L(\wh O_{i}^\pm)}^2\right)
		\end{gather}
		(evidently, the constant $C$ in \eqref{H2est:0} is independent of $i$).
		One has:
		\begin{align}\notag
			\sum_{i\in\Z^n}\|-\Delta f^\pm +f^\pm \|^2_{\L(\wh O_{i}^\pm)}
			&\leq 2^{n-1}\|-\Delta f^\pm +f^\pm \|^2_{\L(\cup_{i\in\Z^n}\wh O_{i}^\pm)}
			\\
			&\leq   
			2^{n-1}\|-\Delta f^\pm +f^\pm \|^2_{\L(\Omega^\pm)}
			\leq 2^{n-1}\|f\|_{\HS^2}^2,\label{H2est:1}
		\end{align}	
		and, similarly,
		\begin{align}	
			\label{H2est:2}
			 \sum_{i\in\Z^n}\|f^\pm \|^2_{\H^1(\wh O_{i}^\pm)}
			\leq 2^{n-1}\|f\|^2_{\HS^1}
		\end{align}
		(the factor $2^n$ comes from the fact that each $x\in \cup_{i\in\Z^n}\wh O_{i}^\pm$ 
		belongs to $2^n$ sets $\overline{\wh O_i^\pm}$).
		Furthermore, introducing the function  $\wh\gamma \in C^1(\R^n)\cap \mathsf{W}^{1,\infty}(\R^n)$ via
		$\wh\gamma (x',x^n)\ceq\gamma(x')$, we get
		\begin{align}\notag 
			\sum_{i\in\Z^n}\|\gamma[f]\|_{\H^{1/2}(\wh\square_{i})}^2 
			&\leq 
			2\sum_{i\in\Z^n}\left(\|\gamma f^+\|^2_{\H^{1/2}(\wh\square_{i})}+\|\gamma f^-\|^2_{\H^{1/2}(\wh\square_{i})} \right)
			\\
			&\notag\leq 
			C\sum_{i\in\Z^n}\left(\|\wh \gamma f^+\|^2_{\H^1(\wh O_{i}^+)}+\|\wh \gamma f^- \|^2_{\H^{1 }(\wh O_{i}^-)}\right)
			\\ 
			&\leq
			2^nC\left(\|\wh\gamma f \|^2_{\H^1(O^+)}+\|\wh\gamma f \|^2_{\H^1(O^-)}\right)\leq 
			C_1\|f\|_{\H^1(\Omega\setminus\Gamma)}^2
			\leq C_1\|f\|_{\HS^1}^2.
			\label{H2est:3} 
		\end{align}
		Combining \eqref{cup:square}, \eqref{H2est:0}--\eqref{H2est:3} and taking into account \eqref{HHH}, we arrive at  
		\eqref{H2est}.
	\end{proof}

	\begin{lemma}
		\label{lemma:main:1}
		One has
		\begin{gather}\label{mainest:1}
			\forall f\in\HS^1 :\quad \|\J\e f - f \|_{\HS} \leq C\mu\e \|f\|_{\HS^1}.
		\end{gather}
	 \end{lemma}
	
	\begin{proof}
		First we estimate $\capty(D\ke)$.
		By virtue of \eqref{capty:inf},
		we get
		\begin{gather*}
			\capty(D\ke)\leq  
			\|\nabla \phi\ke\|^2_{\L(B\ke)}
		\end{gather*}
		(recall that   $\phi\ke$ is given in \eqref{phipsi}).
		Straightforward computations yield  $\|\nabla \phi\ke \|^2_{\L(B\ke)}\leq C(G(d\ke))^{-1}$,
		whence, taking into account \eqref{assump:3}, one arrives at the estimate
		\begin{gather}\label{capty:bound}
			\capty( D\ke )\leq   C\rho\ke^{n-1}\gamma\ke.
		\end{gather}
		
		Let $f\in\HS$. Recall that the sets $B\ke^\pm$, $\wt D\ke^\pm$ are defined in \eqref{Bke}, \eqref{Fke}; note that, due to \eqref{eps0:1+} and \eqref{wtd:cond}, $\wt D\ke^\pm\subset B\ke^\pm\subset \Omega^\pm$. 
		One has $|\phi\ke^\pm|\leq 1$, $|\psi\ke^\pm|\leq 1$, furthermore, we have
		\begin{gather}
			\label{supp:phi}
			\supp(\phi\ke^\pm) \subset \overline{\wt D\ke^\pm},\\
			\label{supp:psi}
			\supp(\psi\ke^\pm) \subset \overline{B\ke^\pm}. 
		\end{gather}
		Using these properties and \eqref{assump:1},
		we obtain
		\begin{align}\notag
			\|\J\e f - f \|_{\L(\Omega^+)}^2
			&=
			\suml_{k\in\N}
			\bigg\|
			\left(\la f^+\ra_{B\ke^+}-f^+ \right) \phi\ke^++
			{1\over 2}U\ke^+\left(\la f^-\ra_{B\ke^-}-\la f^+\ra_{B\ke^+}\right)\psi\ke^+
			\bigg\|^2_{\L(B\ke^+)}
			\\
			&\leq
			\suml_{k\in\N}
			\left( 
			2\|\la f^+\ra_{B\ke^+}-f^+ \|^2_{\L(\wt D\ke^+)}+
			\left(|\la f^-\ra_{B\ke^-}|^2+|\la f^+\ra_{B\ke^+}|^2\right)\|U\ke^+\|^2_{\L(B\ke^+)}
			\right). \label{lemma:prop1:1}
		\end{align}
		Using  the estimates \eqref{Poincare} and \eqref{Fest}, we get
		\begin{align}\notag 
			n\ge 3:\quad\sum_{k\in\N}\|f^+ - \la f^+\ra_{B\ke^+} \|^2_{\L(\wt D\ke^+)}&\leq
			C\sum_{k\in\N}\left(\left({d\ke \over \rho\ke}\right)^n\|f^+-\la f^+\ra_{B\ke^+ } \|^2_{\L(B\ke^+)}+d\ke^2\|\nabla f^+\|^2_{\L(B\ke^+)}\right)\\
			&\leq
			C_1\sum_{k\in\N}\left( {d\ke^{n} \over \rho\ke^{n-2}} +d\ke^2 \right)\|\nabla f^+\|^2_{\L(B\ke^+)}\notag\\
			&\leq
			C_2\sum_{k\in\N}d\ke^2  \|\nabla f^+\|^2_{\L(B\ke^+)} \leq
			C_2\sup_{k\in\N}d\ke^2  \| f\|^2_{\HS^1}  \label{lemma:prop1:2}
		\end{align}
		(in the penultimate estimate we also use  $\rho\ke^{2-n}<d\ke^{2-n}$).
		Similarly, one has
		\begin{align}\label{lemma:prop1:2+}
			n=2:\quad\sum_{k\in\N}\|f^+ - \la f^+\ra_{B\ke^+} \|^2_{\L(\wt D\ke^+)}\leq
			C \sup_{k\in\N}(d\ke \rho\ke|\ln (d\ke \rho\ke)|)  \| f\|^2_{\HS^1}.
		\end{align} 
		Further,  the  estimate \eqref{Friedrichs}  together with 
		\eqref{capty:bound}, \eqref{capty:eq} yield
		\begin{align} 
			\|U\ke^+\|^2_{\L(B\ke^+)}\leq 	\|U\ke\|^2_{\L(B\ke)}
			\leq
			C\rho\ke^2 \|\nabla U\ke\|^2_{\L(B\ke )}  =
			C \rho\ke^2 \capty(D\ke)\leq 
			C_1\rho\ke^{n+1}\gamma\ke.
			\label{lemma:prop1:4}
		\end{align} 
		Using the Cauchy-Schwarz inequality,   \eqref{trace}, \eqref{assump:2} and \eqref{eps0:1},
		we get
		\begin{align}\notag
			|\la f^\pm\ra_{B\ke^\pm}|^2\leq 
			{C \rho\ke^{-n}}\|f^\pm \|^2_{\L(B\ke^\pm)} &\leq 
			{C \rho\ke^{-n}}\left(\rho\ke \|f^\pm\|^2_{\L(S\ke)}+\rho\ke^2 \|\nabla f^\pm\|^2_{\L(B\ke^\pm)}\right)\\ \label{lemma:prop1:3}
			&\leq 
			{C \rho\ke^{1-n}}\left( \|f^\pm\|^2_{\L(S\ke)}+  \|\nabla f^\pm\|^2_{\L(B\ke^\pm)}\right).
		\end{align}
		Combining \eqref{lemma:prop1:4}, \eqref{lemma:prop1:3} and the trace estimate
		\begin{gather}
			\label{trace:standard}
			\|f^\pm\|_{\L(\Gamma)}\leq C\|f^\pm \|_{\H^1(\Omega^\pm)}
		\end{gather}
		we arrive at  
		\begin{multline} 
			\suml_{k\in\N}  \left(|\la f^-\ra_{B\ke^-}|^2+|\la f^+\ra_{B\ke^+}|^2\right)\|U\ke^+\|^2_{\L(B\ke^+)}
			\leq
			C\suml_{k\in\N}\rho\ke^2\gamma\ke\left(\|f^-\|^2_{\L(S\ke)}+\|\nabla f^+\|^2_{\L(B\ke^+)}\right.\\ 
			\left.+\|f^+\|^2_{\L(S\ke)}+\|\nabla f^+\|^2_{\L(B\ke^+)}\right)
			\leq C  \sup_{k\in\N} (\rho\ke^2\gamma\ke)\|f\|^2_{\HS^1}.
			\label{lemma:prop1:4+}
		\end{multline}
		It follows   from  \eqref{lemma:prop1:1}, \eqref{lemma:prop1:2}, \eqref{lemma:prop1:4+} 
		that 
		\begin{gather}\label{lemma:prop1:est+}
			\|\J\e f - f \|_{\L(\Omega^+)} \leq C \zeta\e \|f\|_{\HS^1},
		\end{gather}
			where $\zeta\e$ is given by 
		\begin{gather*}
			\zeta\e\ceq 
			\begin{cases}
				\left( \supl_{k\in\N} {d\ke^2}  + \supl_{k\in\N}({\rho\ke^2}
				{\gamma\ke})\right)^{1/2}     ,&n\ge 3,\\[2mm]
				\left( \supl_{k\in\N}(d\ke \rho\ke |\ln (d\ke\rho\ke)|)+ \supl_{k\in\N}({\rho\ke^2}
				{\gamma\ke})\right)^{1/2}  ,&n=2.
			\end{cases} 
		\end{gather*}
		Repeating verbatim the above arguments  we get similar estimate for $\Omega^-$:
		\begin{gather}\label{lemma:prop1:est-}
			\|\J\e f - f \|_{\L(\Omega^-)} \leq C \zeta\e \|f\|_{\HS^1}.
		\end{gather}
	
		It is easy to show, using \eqref{assump:3}, \eqref{eps0:1}, \eqref{eps0:2}, that
		\footnote{In fact, one has even stronger than \eqref{zeta:eta} property $\zeta\e=o(\sup_{k\in\N}\eta\ke)$, but the knowledge of this fact gives us no profits, since the convergence rate $\sup_{k\in\N}\eta\ke$
			appears later in the estimate for $|\a\e[ \J\e f,u]-\a[f,\J'\e u]|$.}
		\begin{gather}
			\label{zeta:eta}
			\zeta\e \leq  C\sup_{k\in\N}\eta\ke\leq C\mu\e.
		\end{gather}
		The  estimate \eqref{mainest:1} follows immediately from \eqref{lemma:prop1:est+}--\eqref{zeta:eta}.	The lemma is proven.
	\end{proof}

	Now, we start the estimation of the difference 
	$\a\e[ \J\e f,u]-\a[f,\J'\e u],$
	where  $f\in\dom(\A)$, $ u\in\dom(\a\e)$.
	Owing to \eqref{supp:phi}--\eqref{supp:psi} we can represent it as follows:
	\begin{align}\label{a:III} 
		\a\e[ \J\e f,u]-\a[f,\J'\e u]=
		I^{1,+}\e+I^{1,-}\e+I^{2,+}\e+I^{2,-}\e+I^{3}\e.
	\end{align}
	Here
	\begin{align*} 
		&I^{1,\pm}\e \ceq \sum_{k\in\N}
		\left( \nabla\big((\la f^\pm\ra_{B\ke^\pm}-f^\pm )\phi\ke^\pm \big), {\nabla u^\pm}\right)_{\L(\wt D\ke^\pm)},\\  
		&I^{2,\pm}\e \ceq\mp
		{1\over 2}\sum_{k\in\N}f\ke
		\left( \nabla (U\ke^\pm\psi\ke^\pm),{\nabla u^\pm}\right)_{\L(B\ke^\pm)},\\
		&I^{3}\e \ceq -
		\left(\gamma [f],[u]\right)_{\L(\Gamma)}  
	\end{align*}
	(as usual, $f^\pm=f\restriction_{\Omega^\pm}$, $u^\pm=u\restriction_{\Omega^\pm}$,
	$[f]=(f^+-f^-)\restriction_\Gamma$,
	$[u]=(u^+-u^-)\restriction_\Gamma$). 
	\smallskip
	
	To proceed further we require  several 
	properties of the function $U\ke$.
	
	\begin{lemma}One has:
		\begin{gather}\label{H:n1}
			{\partial  U^\pm\ke\over \partial\nu^\pm}=0\text{\quad on }S\ke\setminus \overline{D\ke},\\
			\label{H:n2}
			{\partial  U^+\ke \over\partial\nu^+}=
			{\partial  U^-\ke \over\partial\nu^-}\text{\quad on } D\ke,
			\\
			\label{cap:D} 
			\capty({D\ke})=
			2\int_{D\ke} {\partial U^\pm\ke\over\partial\nu^\pm} \d x'.
		\end{gather} 
		where   $\nu^\pm= \mp(0,\dots,0, 1)$ is the  normal to $\Gamma$ pointed outward of $\Omega^\pm$.
		Furthermore, the following pointwise estimates hold true: 
		\begin{gather}\label{H:estimates:1}
			0\leq U\ke(x)\leq 
			\wt\psi\ke(x)\ceq
			{G(|x-x\ke|)- G(\rho\ke) \over G(d\ke) - G(\rho\ke)},\quad x\in \overline{B\ke\setminus\B(d\ke,x\ke)},
			\\\label{H:estimates:2}\forall j\in \{1,\dots,n\}:\quad
			\left|{\partial U\ke\over\partial x_j}(x)\right|\leq C{\rho\ke^{1-n} \over G(d\ke) - G(\rho\ke)},\quad x\in 
			B\ke\setminus \overline{\B({\rho\ke\over 4},x\ke)}.  
		\end{gather} 
		where the function $G(t)$ is defined by \eqref{Gt}.	
	\end{lemma}
	
	\begin{proof}
		Standard   regularity theory for elliptic PDEs yields
		\begin{gather}
			\label{H:regul}
			U\ke\in  {C}^\infty(B\ke\setminus\overline{D\ke}).
		\end{gather}
		Evidently, $U\ke$ is symmetric 
		with respect to the hyperplane $\Gamma$,
		i.e.
		\begin{gather}
			\label{H:symm}
			U^+\ke(x', \tau)=U\ke^-(x', -\tau),\ \tau>0.
		\end{gather} 
		Then \eqref{H:n1}--\eqref{H:n2} follow immediately from \eqref{H:regul} and \eqref{H:symm}.
		Furthermore, since $U\ke$ is the solution to the problem  \eqref{BVP:cap}, one has the following Green's identity:
		\begin{gather}\label{Green}
			\int_{D\ke}\left({\partial U^+\ke\over\partial\nu^+}+{\partial U^-\ke\over \partial\nu^-}\right)\d x'=
			\|\nabla U\ke\|^2_{\L(B\ke)}.
		\end{gather}
		From \eqref{capty:eq}, \eqref{H:n2} and \eqref{Green}, we infer \eqref{cap:D}.
		
		Now, we proceed to the proof of \eqref{H:estimates:1}.
		By the maximum principle  
		\begin{gather}\label{U01}
			0\le U\ke \leq 1. 
		\end{gather}
		Furthermore, the function $\wt\psi\ke$  is harmonic in $B\ke\setminus\overline{\B(d\ke,x\ke)}$,  it is equal to $0$ on $\partial B\ke$, and it is equal to $1$ on  $\partial \B(d\ke,x\ke)$. Hence, the function 
		$\wt\psi\ke-U\ke$ is harmonic in $B\ke\setminus\overline{\B(d\ke,x\ke)}$, equal to $0$ on 
		$\partial B\ke$ and is non-negative on $\partial \B(d\ke,x\ke)$.
		Again applying the maximum principle, we conclude that 
		\begin{gather}\label{wtpsi:U}
			\wt\psi\ke-U\ke\geq 0\text{ in }\overline{B\ke\setminus\B(d\ke,x\ke)}.
		\end{gather}
		The estimate \eqref{H:estimates:1} follows from \eqref{U01} and \eqref{wtpsi:U}.
		
		Finally, we prove \eqref{H:estimates:2}. 
		Let $x\in  B\ke\setminus \overline{\B({\rho\ke\over 4},x\ke)}$.
		We denote $\tau_x\ceq \rho\ke^{-1}|x-x\ke|$ (that is $\tau_x\in (1/4,1)$),
		and $\ell_x\ceq \min\{{1\over 8};\,1-\tau_x\}$.
		Due to \eqref{eps0:2} one has
		$$\B(\ell_x\rho\ke,x)\subset B\ke\setminus\overline{\B(d\ke,x\ke)}.$$
		Since $U\ke$ is a harmonic function, its partial derivatives are harmonic functions too.
		Then, using the mean value theorem 
		for harmonic functions and then   integrating by parts, we get
		\begin{gather}\label{mvt}
			{\partial U\ke\over\partial x^j}(x)= 
			{1\over \vol(\B(\ell_x\rho\ke,x))}\int_{ \B(\ell_x\rho\ke,x)}{\partial U\ke\over\partial x^j}(y) \d y
			={1\over \vol(\B(\ell_x\rho\ke,x))}\int_{\partial \B(\ell_x\rho\ke,x)}u(y)\nu_j(y) \d s_y,
		\end{gather}
		where 
		$\d s_y$ is a element of integration on the sphere ${\partial \B(\ell_x\rho\ke,x)}$,  
		$\nu_j(y)$ is the $j$-th component of the outward pointing unit normal on $\partial \B(\ell_x\rho\ke,x)$ (thus $|\nu_j(y)|\leq 1$).
		Using \eqref{H:estimates:2} we infer from  \eqref{mvt}:
		\begin{align*}
			\left|{\partial U\ke\over \partial x^j}(x)\right| 
			&\leq C(\ell_x\rho\ke)^{-1}(G(d\ke)-G(\rho\ke))^{-1}\max_{y\in \partial \B(\ell_x,x)}(G(|y-x\ke|)-G(\rho\ke))
			\\
			&=C\rho\ke^{1-n}(G(d\ke)-G(\rho\ke))^{-1}\vartheta_x,
		\end{align*}
		where 
		$$
		\vartheta_x\ceq 
		\begin{cases}
			\left((\tau_x-\ell_x)^{2-n}-1\right)\ell_x^{-1},& n\ge 3,
			\\-\ln(\tau_x-\ell_x)\ell_x^{-1},
			& n=2.
		\end{cases}
		$$
		(here we use the fact that   $\max_{y\in \partial \B(\ell_x\rho\ke,x)}G(|y-x\ke|)$ is attained 
		at the point $y_0\in {\partial \B(\ell_x\rho\ke,x)}$
		lying on the interval between $x$ and $x\ke$, that is $|y_0-x\ke|=(\tau_x-\ell_x)\rho\ke$).
		It is easy to show that   $\vartheta_x$ is uniformly bounded
		on $\{ (\tau_x,\ell_x):\ \tau_x\in (1/4,1),\ \ell_x=\min\{{1\over 8};\,1-\tau_x\}\}$.
		The estimate \eqref{H:estimates:2} is proven.
	\end{proof}
	
	Now, we can further transform the terms $I\ke^{2,\pm}$. Integrating by parts twice and taking into account \eqref{H:n1}, \eqref{cap:D} and the properties
	\begin{gather*}
		\psi\ke=1\text{ in a neighborhood of }D\ke,\quad
		\psi\ke=0\text{ in a neighborhood of }\partial B\ke,\quad
		{\partial \psi\ke^\pm\over\partial\nu^\pm}=0\text{ on }\Gamma,
	\end{gather*}
	we get:
	\begin{align*}\notag
		I\e^{2,\pm}
		&=
		\pm{1\over 2}\sum_{k\in\N}f\ke
		(\Delta(U\ke^\pm\psi\ke^\pm),  u^\pm  )_{\L(B\ke^\pm)}+Q^{1,\pm}\e
		\\\notag
		&=
		\pm{1\over 2}\sum_{k\in\N}f\ke
		(\Delta(U\ke^\pm\psi\ke^\pm), \la u^\pm\ra_{B\ke^\pm} )_{\L(B\ke^\pm)}+Q\e^{1,\pm}+Q^{2,\pm}\e
		\\\notag
		&=
		\pm{1\over 2}\sum_{k\in\N}f\ke\overline{\la u^\pm\ra_{B\ke^\pm}} \left( \int_{D\ke }
		{\partial U\ke^\pm\over \partial\nu^\pm} \d x'\right)+Q\e^{1,\pm}+Q^{2,\pm}\e
		\\ \notag
		&=
		\pm  {1\over 4}\sum_{k\in\N}f\ke\overline{\la u^\pm\ra_{B\ke^\pm}} \capty(D\ke)+Q\e^{1,\pm}+Q^{2,\pm}\e
		\\
		&=
		Q\e^{1,\pm}+Q^{2,\pm}\e
		+Q^{3,\pm}\e	+Q^{4,\pm}\e ,
	\end{align*} 
	where  
	\begin{align*}
		Q\e^{1,\pm}&\ceq \mp{1\over 2}\sum_{k\in\N}f\ke \int_{D\ke}
		{\partial U\ke^\pm\over\partial\nu^\pm}\overline{u^\pm} \d x',
		\\
		Q^{2,\pm}\e&\ceq \pm{1\over 2}\sum_{k\in\N}f\ke
		(\Delta(U\ke^\pm\psi\ke^\pm),{u^\pm-\la u^\pm \ra_{B\ke^\pm} })_{\L(B\ke^\pm)},\\
		Q^{3,\pm}\e&\ceq   
		\pm  {1\over 4}\sum_{k\in\N}\left(
		f\ke\overline{\la u^\pm\ra_{B\ke^\pm}} -
		\la [f]\ra_{S\ke}\overline{\la u^\pm\ra_{S\ke }}\right) \capty(D\ke),\\
		Q^{4,\pm}\e&\ceq \pm  {1\over 4}\sum_{k\in\N}\la [f]\ra_{S\ke}\overline{\la u^\pm\ra_{S\ke }} \capty(D\ke)
	\end{align*}
	Denoting
	$$I^{4}\e\ceq Q^{4,+}\e+ Q^{4,-}\e+ I^3\e,$$
	 we can rewrite \eqref{a:III}  as follows:
	\begin{gather}\label{aa:repres}
		\a\e[\J\e f,u]-\a[f,\J'\e u] =
		I^{1}\e+
		\sum_{k=1}^3\left(Q^{k,+}\e+ Q^{k,-}\e\right)
		+I^{4}\e.
	\end{gather}
	
	By virtue of \eqref{H:n2} and the fact that $u$ is continuous across $D\ke$, we immediately get
	\begin{gather}
		\label{Q1pm}
		Q\e^{1,+}+Q\e^{1,-}=0.
	\end{gather}
	In Lemmata~\ref{lemma:I1}--\ref{lemma:I4} below we estimate the other terms in the right-hand-side of \eqref{aa:repres}. Recall that  $\eta\ke$ is defined by \eqref{etake}.
	
	\begin{lemma}\label{lemma:I1}
		One has
		\begin{gather}\label{I1:est}
			|I_{\eps}^{1,\pm}|\leq C\sup_{k\in\N}\eta\ke\|f\|_{\HS^2}\|u\|_{\HS\e^1} .
		\end{gather}
	\end{lemma}
	
	\begin{proof}
		Using $|\phi^\pm\ke|\leq 1$  and the Cauchy-Schwarz inequality, we get 
		\begin{align}  
			|I_{\eps}^{1,\pm}|&=
			\left|\suml_{k\in\N}\left(-(\phi\ke^\pm \nabla f^\pm ,\nabla u^\pm)_{\L(\wt D\ke^\pm)} + 
			((\la f^\pm\ra_{B^\pm\ke}-f^\pm ) \nabla \phi\ke^\pm,\nabla u^\pm)_{\L(\wt D\ke^\pm)}\right)\right| \notag \\ 
			&\leq 
			\left\{ 
			\Bigl( {\suml_{k\in \N}\|\nabla f^\pm \|_{\L(\wt D\ke^\pm)}^2\Bigr)^{1/2}}  + \Bigl(
			{\suml_{k\in \N}\|( f^\pm - \la f^\pm\ra_{B^\pm\ke})\nabla \phi\ke^\pm\|_{\L(\wt D\ke^\pm)}^2\Bigr)^{1/2}} \right\}\|u\|_{\HS\e^1}.\label{I1:1}
		\end{align} 
		
		Using   Lemma~\ref{lemma:Fest},
		we obtain
		\begin{gather}\label{I1:2} 
			\|\nabla f \| _{\L(\wt D\ke^\pm)}\leq  
			C\wt\eta\ke\|f^\pm\| _{\H^2(B\ke^\pm)} ,
		\end{gather}
		where $\wt\eta\ke$ is defined by
		\begin{gather*}
			\wt\eta\ke\ceq 
			\begin{cases}
				\max\left\{d\ke^{n/2}\rho\ke^{-n/2};\,d\ke \right\},&n\ge 3,\\[2mm]
				\max\left\{d\ke^{1/2} \rho\ke^{-1/2};\, (d\ke\rho\ke|\ln (d\ke\rho\ke)|)^{1/2}\right\},&n=2.
			\end{cases}
		\end{gather*}
		It is easy to see that $\wt\eta\ke\leq \eta\ke$, whence,
        using \eqref{I1:2}, \eqref{assump:1},  
		$B\ke^\pm\subset O^\pm$ (this follows from \eqref{eps0:1+}) and Lemma~\ref{lemma:H2}, we get
		\begin{gather}\label{I1:3}
			\left(\suml_{k\in\N}\|\nabla f^\pm\|^2_{\L(\wt D\ke^\pm)}\right)^{1/2}\leq 
			C\sup_{k\in\N}\wt\eta\ke\|f \|_{\H^2(O^\pm)}\leq 
			C_1 \sup_{k\in\N} \eta\ke
			\| f\|_{\HS^2}.
		\end{gather}
		
		To estimate the second term in the right-hand-side of \eqref{I1:1},
		we define the numbers $p,\, q$ via
		\begin{align}\label{pmax}
			&p=\frac {2n}{n-4} \text{\, if }n\geq 5,\quad
			&&p=2|\ln d\ke -\ln \rho\ke |\text{\, if }n=4,\quad
			&&p=\infty\text{\, if }n=2,3,\\\label{qmax}
			&q={n\over 2}\text{\, if }n\geq 5,\quad 
			&&q={2\over 1-|\ln d\ke -\ln \rho\ke |^{-1}}\text{\, if }n=4,\quad 
			&&q=2\text{\, if }n=2,3.
		\end{align}
		Note that, due to \eqref{eps0:2}, we have $|\ln d\ke -\ln \rho\ke |\le  \ln 8>1$.
		It is easy to see that $p,q\in[2,\infty]$ and ${1\over p}+ {1\over q}={1\over 2}$. Then, 
		by virtue of the H\"older inequality, we get
		\begin{align} \notag
			\|(f^\pm -\la f^\pm\ra_{B^\pm\ke})\nabla  \phi\ke^\pm\|^2_{\L(\wt D\ke^\pm)}&\leq 
			\|(f^\pm -\la f^\pm\ra_{B^\pm\ke})\nabla  \phi\ke^\pm\|^2_{\L(B\ke^\pm)}\\&\leq 
			\|f^\pm - \la f^\pm\ra_{B^\pm\ke}\|^2_{\LL^p(B\ke^\pm)} \|\nabla  \phi\ke^\pm\|^2_{\LL^q(B\ke^\pm)}.\label{Hoelder}
		\end{align}  
		Using Lemma~\ref{lemma:sobolev} for $v\ceq f^\pm -\la f^\pm\ra_{B^\pm\ke}$ and $p$ as in \eqref{pmax} and taking into account \eqref{lemma:c4p:est}, we obtain the estimate  
		\begin{gather}\label{pmax+}
			\|f^\pm -\la f^\pm\ra_{B^\pm\ke}\|_{\LL^p(B^\pm\ke)}\leq C
			\|f^\pm \|_{\mathsf{H}^2(B^\pm\ke)}
			\begin{cases}
				\rho\ke^{-1}&n\geq 5,\\
				|\ln d\ke - \ln \rho\ke |\cdot\rho\ke^{2|\ln d\ke - \ln \rho\ke |^{-1}-1}&n= 4,\\
				\rho\ke^{-1/2}&n= 3,\\
				1&n= 2.\\
			\end{cases}
		\end{gather} 
		For $q$ as in \eqref{qmax} we get via  straightforward calculations:
		\begin{gather}
			\label{wtphi-est}
			\|\nabla \phi\ke^\pm\|_{\LL^q(B\ke^\pm)}\leq 
			C\begin{cases}
				d\ke,&n\ge 5,\\
				d\ke^{1-2|\ln {d\ke}-\ln{\rho\ke}|^{-1}},&n=4,\\
				d\ke^{1/2},&n=3,\\
				|\ln {d\ke}-\ln{\rho\ke}|^{-1/2},&n=2.
			\end{cases} 
		\end{gather} 
		Combining \eqref{Hoelder}--\eqref{wtphi-est} and taking into account that 
		$\left({d\ke\over\rho\ke}\right)^{-2|\ln d\ke - \ln\rho\ke|^{-1}}=\exp(2)$,
		we get
		\begin{gather}\hspace{-2mm}
			\|(f^\pm-\la f^\pm\ra_{B^\pm\ke})\nabla \phi\ke^\pm\|_{\L(\wt D\ke^\pm)}\leq  C \eta\ke
			\|f^\pm\|_{\H^2(B\ke^\pm)},
			\label{I1:4}
		\end{gather} 
		where $\eta\ke$ is given by \eqref{etake}.
		Taking into account \eqref{H2est}, we deduce from \eqref{I1:4}:
		\begin{gather}\label{I1:5}
			\left(\suml_{k\in\N}\|(f^\pm-\la f^\pm\ra_{B^\pm\ke})\nabla \phi\ke^\pm\|_{\L(\wt D\ke^\pm)}^2\right)^{1/2}\leq  C \sup_{k\in\N}\eta\ke
			\|f\|_{\HS^2}.
		\end{gather}   
		Combining \eqref{I1:1}, \eqref{I1:3}, \eqref{I1:5},  
		we arrive at the required estimate \eqref{I1:est}.
		The lemma is proven.
	\end{proof}

	\begin{lemma}\label{lemma:Q2}
		One has
		\begin{gather}\label{Q2:est}
			|Q^{2,\pm}\e|\le C\sup_{k\in\N}(\rho\ke^{1/2}\gamma\ke\varkappa\ke)\|f\|_{\HS^1}\|u\|_{\HS^1\e},
		\end{gather}
		where
		\begin{gather*} 
			\varkappa\ke\ceq
			\begin{cases}
				1,&n\ge 3,\\
				|\ln \rho\ke|,&n=2.
			\end{cases}
		\end{gather*}
	\end{lemma}
	
	\begin{proof}
		Using the Cauchy-Schwarz inequality, we get
		\begin{align}\notag\hspace{-5mm}
			|Q^{2,+}\e|\le 
			\left(\sum_{k\in\N}\rho\ke\gamma\ke^{-2}\varkappa\ke^{-2}|f\ke|^2  
			\|\Delta (U\ke^+\psi\ke^+)\|^2_{\L(B\ke^+)}\right)^{1/2}&\\\times
			\left( \sum_{k\in\N}\rho\ke^{-1}\gamma\ke^{2}\varkappa\ke^{2}\|u^+-
			\la u^+\ra_{B\ke^+}\|^2_{\L(B\ke^+)}\right)^{1/2}&.\label{lemma:Q2:1}
		\end{align} 
		Since
		$\Delta U\ke=0$, we have
		\begin{gather}
			\label{AHphi}
			\Delta(U\ke^+\psi\ke^+)=2(\nabla U\ke^+,\nabla\psi\ke^+)+U\ke^+\Delta \psi\ke^+.    
		\end{gather}
		Furthermore,  one has
		\begin{gather}\label{supp:nablapsi}
			\supp(\nabla \psi\ke^\pm)\cup\supp(\Delta \psi\ke^\pm) \subset 
			\overline{B\ke^\pm\setminus\B({\rho\ke\over 4},x\ke)},\\
			\label{psi:est}
			|\psi\ke^\pm (x)|\leq 1,\quad |\nabla \psi\ke^\pm (x)|\leq C\rho\ke^{-1},\quad 
			|\Delta \psi\ke^\pm (x)|\leq C\rho\ke^{-2}.
		\end{gather}
		Then, using \eqref{H:estimates:1}--\eqref{H:estimates:2} and \eqref{AHphi}--\eqref{psi:est},
		we conclude  
		\begin{align}\notag
			\|\Delta(U\ke^+\psi\ke^+)\|^2_{\L(B\ke^+)} 
			&\leq
			C\rho\ke^n\|2( \nabla U\ke^+,\nabla\psi\ke^+)+U\ke^+\Delta \psi\ke^+\|^2_{\LL^\infty(B\ke^+\setminus\B({\rho\ke\over 4},x\ke))}\\\notag&\leq
			C_1\rho\ke^{n }\left(\rho\ke^{ -2}\|\nabla U\ke^+\|^2_{\LL^\infty(B\ke^+\setminus\B({\rho\ke\over 4},x\ke))}+
			\rho\ke^{ -4}\|U\ke^+\|^2_{\LL^\infty(B\ke^+\setminus\B({\rho\ke\over 4},x\ke))}\right)
			\\
			&\leq C_2\rho\ke^{n-2}\gamma\ke^2\varkappa\ke^2.\label{est:div}
		\end{align}
		Combining \eqref{lemma:prop1:3}, \eqref{trace:standard}, \eqref{est:div} 
		we arrive at 
		\begin{align} \notag
			&\sum_{k\in\N}\rho\ke\gamma\ke^{-2}\varkappa\ke^{-2}|f\ke|^2\|\Delta (U\ke^+\psi\ke^+) \|^2_{\L(B\ke^+)}
			\\\notag
			& \leq 2\sum_{k\in\N}\rho\ke\gamma\ke^{-2}\varkappa\ke^{-2}\left(|\la f^-\ra_{B\ke^-}|^2+|\la f^+\ra_{B\ke^+}|^2\right)\|\Delta (U\ke^+\psi\ke^+) \|^2_{\L(B\ke^+)}\\
			&  \leq 
			C\sum_{k\in\N}\left(\|f^+\|^2_{\L(S\ke)}+  \|\nabla f^+\|^2_{\L(B\ke^+)}+
			\|f^-\|^2_{\L(S\ke)}+  \|\nabla f^-\|^2_{\L(B\ke^-)}\right)
			\le C_1\|f\|^2_{\HS^1}.\label{Q2:term1}
		\end{align}
		
		Finally, by   using   \eqref{Poincare}, we estimate the second factor in \eqref{lemma:Q2:1} :
		\begin{align}\notag
			\sum_{k\in\N}\rho\ke^{-1}\gamma\ke^{2}\varkappa\ke^{2}\|u-
			\la u\ra_{B\ke^+}\|^2_{\L(B\ke^+)}&\leq 
			C\sum_{k\in\N}\rho\ke\gamma\ke^{2}\varkappa\ke^{2}\|\nabla u\|^2_{\L(B\ke^\pm)}
			\\
			&\leq
			C(\sup_{k\in\N}\rho^{1/2}\ke\gamma\ke \varkappa\ke)^2\|u\|^2_{\HS^1}. 
			\label{Q2:term2} 
		\end{align}
		The required estimate \eqref{Q2:est} for $Q\e^{2,+}$ follows from \eqref{lemma:Q2:1}, \eqref{Q2:term1}, \eqref{Q2:term2}. For $Q\e^{2,-}$ the proof is similar.
		The lemma is proven.
	\end{proof}
	
	\begin{lemma}
		\label{lemma:Q3}
		One has
		\begin{gather}\label{Q3:est}
			|Q^{3,\pm}\e|\leq C\sup_{k\in\N}(\rho\ke^{1/2}\gamma\ke)\|f\|_{\HS^1}\|u\|_{\HS\e^1}.
		\end{gather}
	\end{lemma}
	
	\begin{proof}One can represent $Q^{3,+}\e$ is the form $Q^{3,+}\e=Q^{3,1 }\e+Q^{3,2 }\e$, where
		\begin{align*}
			Q^{3,1 }\e &=
			{1\over 4}\sum_{k\in\N}
			\left\{\left(\la f^+\ra_{B\ke^+}-\la f^+\ra_{S\ke}\right)-\left(\la f^-\ra_{B\ke^-}-\la f^-\ra_{S\ke}\right)\right\} \overline{\la u^+\ra_{B\ke^+}}\capty(D\ke) \\
			Q^{3,2 }\e &=
			{1\over 4}\sum_{k\in\N}
			\left(\la f^+\ra_{S\ke}-\la f^-\ra_{S\ke}\right)\overline{(\la u^+\ra_{B\ke^\pm}-\la u^+\ra_{S\ke})}\capty(D\ke) 
		\end{align*}		
		Using   \eqref{mean:mean} (applied for $f^\pm$), \eqref{lemma:prop1:3} and \eqref{trace:standard} (applied for $u^+$), \eqref{capty:bound}, we get	
		\begin{multline}\label{Q31}
			|Q^{3,1 }\e|\leq 
			{1\over 4}\left(2\sum_{k\in\N}(\capty(D\ke))^2 \rho\ke ^{1-n}\left(|\la f^+\ra_{B\ke^+}-\la f^+\ra_{S\ke}|^2+ |\la f^-\ra_{B\ke^-}-\la f^-\ra_{S\ke}|^2\right)    \right)^{1/2}
			\\
			\times\left(\sum_{k\in\N} \rho\ke ^{n-1}|\la u^+\ra_{B\ke^+}|^2 \right)^{1/2}
			\leq 
			C\left(\sum_{k\in\N\e}\rho\ke\gamma\ke^2\left(\|\nabla f^+\|^2_{\L(B\ke^+)}+\|\nabla f^-\|^2_{\L(B\ke^-)}\right) \right)^{1/2}
			\\
			\times
			\left(\sum_{k\in\N}
			\| \nabla u^+\|^2_{\L(B\ke^+)}+ \sum_{k\in\N}\|   u^+\|^2_{\L(S\ke)}  \right)^{1/2}\leq
			C\sup_{k\in\N}(\rho\ke^{1/2}\gamma\e)\|f\|_{\HS^1} \|u\|_{\HS^1\e}.
		\end{multline}
		Similarly, applying the Cauchy-Schwarz inequality (for $f^\pm$),  \eqref{mean:mean} (for $u^+$),  and \eqref{capty:bound}, we estimate $Q^{3,2 }\e$:
		\begin{multline}\label{Q32}
			|Q^{3,2 }\e|\leq 
			{1\over 4}\left(2\sum_{k\in\N}(\capty(D\ke))^2 \rho\ke ^{2-n}\left(|\la f^+\ra_{S\ke}|^2+ 
			|\la f^-\ra_{S\ke}|^2\right)    \right)^{1/2}
			\\
			\times\left(\sum_{k\in\N} \rho\ke ^{n-2}|\la u^+\ra_{B\ke^+}-\la u^+\ra_{S\ke}|^2 \right)^{1/2}
			\leq 
			C\left(\sum_{k\in\N\e}\rho\ke\gamma\ke^2\left(\| f^+\|^2_{\L(S\ke)}+\|  f^-\|^2_{\L(S\ke )}\right) \right)^{1/2}
			\\
			\times
			\left(\sum_{k\in\N}
			\| \nabla u^+\|^2_{\L(B\ke^+)}  \right)^{1/2}\leq
			C\sup_{k\in\N}(\rho\ke^{1/2}\gamma\e)\|f\|_{\HS^1} \|u\|_{\HS^1\e}.
		\end{multline}	
		Combining \eqref{Q31} and \eqref{Q32} we arrive at the required estimate \eqref{Q3:est} for
		$Q^{3,+}\e$. For $Q\e^{3,-}$ the proof is similar.
		The lemma is proven.
	\end{proof}

	\begin{lemma}\label{lemma:I4}
		One has
		\begin{gather}\label{I4:est}
			|I\e^4|\leq C\kappa\e\|f\|_{\HS^2}\|u\|_{\HS\e^1},
		\end{gather}
		where $\kappa\e\to 0$ is given in \eqref{assump:main}.
	\end{lemma}
	
	\begin{proof}
		Since $f^\pm\in\H^2(O^\pm)$ (see Lemma~\ref{lemma:H2}) and $u^\pm\in\H^1(\Omega^\pm)$,
		the trace theorem yields $f^\pm\in\H^{3/2}(\Gamma)$ and $u^\pm\in\H^{1/2}(\Gamma)$, moreover
		\begin{gather}\label{traces321}
			\|f^\pm\|_{\H^{3/2}(\Gamma)}\leq C\|f^{\pm}\|_{\H^2(O^\pm)} ,\quad
			\|u^\pm\|_{\H^{1/2}(\Gamma)}\leq C\|u^{\pm}\|_{\H^1(O^\pm)} .
		\end{gather}
		Then, using  \eqref{assump:main}, \eqref{traces321} and \eqref{H2est}, we obtain
		\begin{align*}
			|I\e^{4,+}|&=
			\left|{1\over 4}\sum_{k\in\N}\la [f]\ra_{S\ke}\overline{\la [u]\ra_{S\ke}} \capty(D\ke)
			-\left(\gamma [f],[u]\right)_{\L(\Gamma)}  \right|
			\leq
			\kappa\e\|[f]\|_{\H^{3/2}(\Gamma)} \|[u]\|_{\H^{1/2}(\Gamma)}\\&\leq
			\kappa\e
			\left(\|f^+\|_{\H^{3/2}(\Gamma)}+\|f^-\|_{\H^{3/2}(\Gamma)}\right)
			\left(\|u^+\|_{\H^{1/2}(\Gamma)}+\|u^-\|_{\H^{1/2}(\Gamma)}\right)
			\\
			&\leq
			C\kappa\e 
			\left(\|f^+\|_{\H^2(O^+)}+\|f^-\|_{\H^2(O^-)}\right) 
			\left(\|u^+\|_{\H^1(O^+)}+\|u^-\|_{\H^1(O^-)}\right) 
			\leq 
			C_1\kappa\e 
			\|f\|_{\HS^2}\|u\|_{\HS^1\e}.
		\end{align*}
		The lemma is proven.  
	\end{proof}
	
	Using the definitions of $\eta\ke$ and $\gamma\ke$ one can easily get
	\begin{gather}\label{rhoga:1}
		\rho\ke^{1/2}\gamma\ke=\eta\ke \gamma\ke^{1/2} A\ke,\text{ where }A\ke\ceq 
		\begin{cases}
			\left({d\ke\over\rho\ke}\right)^{n-4\over 2},&n\ge 5,\\
			\left|\ln {d\ke\over\rho\ke}\right|^{-1},&n=4,\\
			1,&n=3,\\
			\left|1+\gamma\ke\rho\ke{\ln \rho\ke }\right|^{1/2},&n=2.
		\end{cases}
	\end{gather}
	Note that 
	\begin{gather}\label{Ake:prop}
		A\ke\leq C
	\end{gather}
	(actually, we even have $\sup_{k\in\N}A\ke\to 0$ as $\eps\to0$ for $n\ge 4$).
	The equality \eqref{rhoga:1} shows that the estimate \eqref{I1:est}
	provides worse (or the same) convergence rate than the estimate \eqref{Q3:est} and, if $n\ge 3$,
	than  the estimate \eqref{Q2:est}. In the case $n=2$ one has a ``competition'' between the converge rates $\sup_{k\in\N}\eta\ke$ coming from  \eqref{I1:est} and $\sup_{k\in\N}(\rho\ke^{1/2}\gamma\ke|\ln\rho\ke|)$ coming from \eqref{Q2:est}:
	for example, if $\gamma\ke\geq C_1>0$, then the second one dominates, but for very small  $d\ke$ (namely, if $\sup_{k\in\N}(\gamma\ke^{1/2}|\ln\rho\ke|)\to 0$ as $\eps\to 0$)
	the first one gives worse convergence rate. Taking into account the above observations, \eqref{HHH} and the definition \eqref{etae} of $\mu\e$, 
	we conclude from
	 \eqref{aa:repres}--\eqref{I1:est}, \eqref{Q2:est}, \eqref{Q3:est}, \eqref{I4:est}, 
	 the final estimate for   $\a\e[\J\e f,u]-\a[f,\J'\e u]$.
	
	\begin{lemma}
		\label{lemma:main:2}
		One has
		\begin{gather}\label{mainest:2}
			\forall f\in\HS^2,\, u\in\HS^1 :\  
			|\a\e[\J\e f,u]-\a[f,\J'\e u]|\leq 
			C\mu\e\|f\|_{\HS^2}\|u\|_{\HS^1\e}.
		\end{gather}
	\end{lemma} 
	
	\subsection{End of proofs of Theorems~\ref{th1} and \ref{th2}}\label{subsec:3:4}
	
	By virtue of Proposition~\ref{propA1}, the properties 
	\eqref{automatic}, \eqref{mainest:1}, \eqref{mainest:2} imply
	the estimate
	\begin{align}\notag\forall f\in\HS:\quad
		\|(\A\e+\Id)^{-1}f -\J\e (\A+\Id)^{-1}f\|_{\HS }&\leq \|(\A\e+\Id)^{-1}f -\J\e (\A+\Id)^{-1}f\|_{\HS^1\e}\\\label{diffAA}&\leq C\mu\e\|f\|_{\HS}
	\end{align}
	(the first inequality above   follows trivially from the definition of $\|\cdot \|_{\HS^1\e}$)
	Then, using \eqref{HHH}, \eqref{mainest:1}, \eqref{diffAA} we get  
	\begin{align*}
		\|(\A\e+\Id)^{-1}f - (\A+\Id)^{-1}f\|_{\HS}&\leq
		\|(\A\e+\Id)^{-1}f -\J\e (\A+\Id)^{-1}f\|_{\HS}+
		\|(\J\e-\Id)(\A+\Id)^{-1}f\|_{\HS}
		\\
		&\leq C\mu\e\left(\|f\|_{\HS}+\|(\A+\Id)^{-1}f\|_{\HS^1}\right)\\
		&\leq C\mu\e\left(\|f\|_{\HS}+\|(\A+\Id)^{-1}f\|_{\HS^2}\right)
		= 2C\mu\e\|f\|_{\HS},
	\end{align*}
	and Theorem~\ref{th1} is proven. 
	
	Now, we proceed to the proof of Theorem~\ref{th2}.
	Let $f\in\HS$ and $g=(\A+\Id)^{-1}f$. One has
	\begin{gather}\label{th2:proof1}
		(\A\e+\Id)^{-1}f -(\A+\Id)^{-1} f - \mathscr{K}\e f=
		(\A\e+\Id)^{-1}f -\J\e(\A+\Id)^{-1} f +(\mathscr{V}\e +\mathscr{W}\e)g.
	\end{gather}
	Here $\mathscr{K}\e$ is defined by \eqref{Ke}, and  
	\begin{align*}
		(\mathscr{V}\e g)^\pm (x)&\ceq 
		\ds\sum_{k\in\N}\left(\la g^\pm \ra_{B\ke^\pm}-g^\pm(x)\right)\phi^\pm\ke(x),\ x\in\Omega^\pm,
		\\
		(\mathscr{W}\e g)^\pm (x)&\ceq 
		\pm\ds{1\over 2}\sum_{k\in\N} 
		\la[g]\ra_{S\ke} U\ke^\pm(x)\mp{1\over 2}\sum_{k\in\N}g\ke U\ke^\pm(x)\psi\ke^\pm(x) ,\
		x\in\Omega^\pm,
	\end{align*}
	where $g\ke\ceq \la g^+ \ra_{B\ke^+}-\la g^- \ra_{B\ke^-}$.
	Due to   \eqref{diffAA} we have
	\begin{gather}\label{diffAA+}
		\|(\A\e+\Id)^{-1}f -\J\e(\A+\Id)^{-1} f\|_{\H^1(\Omega\setminus\Gamma)}=
		\|(\A\e+\Id)^{-1}f -\J\e(\A+\Id)^{-1} f\|_{\HS^1\e}\leq C_1\|f\|_{\HS}.
	\end{gather}
	
	Using the estimates \eqref{lemma:prop1:2}, \eqref{lemma:prop1:2+}, \eqref{I1:3}, \eqref{I1:5} (applied for $g^\pm$ instead of $f^\pm$) and  \eqref{zeta:eta}, we get
	\begin{align}
		\label{Ve:est:1}
		\| (\mathscr{V}\e g)^\pm \|_{\L(\Omega^\pm)}&\leq
		\Bigl(\suml_{k\in \N}\|g^\pm - \la g^\pm\ra_{B^\pm\ke}\|_{\L(\wt D\ke^\pm)}^2\Bigr)^{1/2}
		\leq C\zeta\e\|g\|_{\HS^1}\leq C_1\mu\e\|g\|_{\HS^1},
		\\
		\notag
		\|\nabla ( \mathscr{V}\e g)^\pm  \|_{\L(\Omega^\pm)}
		&\leq  
		\Bigl( {\suml_{k\in \N}\|\nabla g^\pm \|_{\L(\wt D\ke^\pm)}^2\Bigr)^{1/2}}  + \Bigl(
		{\suml_{k\in \N}\|( g^\pm - \la g^\pm\ra_{B^\pm\ke})\nabla \phi^\pm\ke\|_{\L(\wt D\ke^\pm)}^2\Bigr)^{1/2}}
		\\\label{Ve:est:2}
		&\leq C\sup_{k\in\N}\eta\ke\|g\|_{\HS^2}\leq C \mu\e\|g\|_{\HS^2}.
	\end{align}
	
	Let us now estimate $(\mathscr{W}\e g)^+$. 
	One has
	\begin{align*}
		(\mathscr{W}\e g)^+ =\ds
		{1\over 2}\suml_{k\in \N}\left(\la g^+\ra_{S\ke}-\la g^+\ra_{B\ke^+}\right) U\ke^+-{1\over 2}\suml_{k\in \N}\left(\la g^-\ra_{S\ke}-\la g^-\ra_{B\ke^-}\right) U\ke^+ 
		\\
		+{1\over 2}\suml_{k\in \N}\left(\la g^+\ra_{B\ke^+}-\la g^-\ra_{B\ke^-}\right) U\ke^+ (1-\psi\ke^+ ).
	\end{align*}
	whence, taking into account \eqref{assump:1}, we infer
	\begin{align}\notag
		\| (\mathscr{W}\e g)^+ \|_{\H^1(\Omega^+)}&\leq 
		{1\over 2}\left(\sum_{k\in\N}|\la g^+\ra_{S\ke}-\la g^+\ra_{B\ke^+}|^2\|U\ke^+\|^2_{\H^1(B\ke^+)}\right)^{1/2}\\\notag
		&+
		{1\over 2}\left(\sum_{k\in\N}|\la g^-\ra_{S\ke}-\la g^-\ra_{B\ke^-}|^2\|U\ke^+\|^2_{\H^1(B\ke^+)}\right)^{1/2}\\
		&+
		{1\over 2}\left(2\suml_{k\in \N}
		\left(|\la g^+\ra_{B^+\ke}|^2+|\la g^-\ra_{B\ke^-}|^2\right)\|U\ke^+(1-\psi\ke^+)\|^2_{\H^1(B\ke^+)}\right)^{1/2}.	\label{We:1}
	\end{align}	
	From \eqref{lemma:prop1:4} we get
	\begin{gather}
		\|U\ke^+ \|_{\H^1(B\ke^+)}^2=\|\nabla U\ke^+\|_{\L(B\ke^+)}^2+ \|U\ke^+\|_{\L(B\ke^+)}^2\leq   C\rho\ke^{n-1}\gamma\ke.\label{We:2}
	\end{gather}
	Also, using the properties  
	$$
	0\leq \psi^+\ke\leq 1,
	\quad 
	|\nabla\psi\ke^+|\leq C\rho\ke^{-1},
	\quad
	\supp(1-\psi\ke)\cap B\ke^+ \subset \overline{B\ke^+\setminus\B({\rho\ke\over 4},x\ke)}
	$$
	and
	\eqref{H:estimates:1}--\eqref{H:estimates:2},    we deduce easily the estimate
	\begin{multline}
		\|U\ke^+(1-\psi\ke^+)\|_{\H^1(B\ke^+)}^2=
		\|\nabla(U\ke^+(1-\psi\ke^+))\|_{\L(B\ke^+\setminus\B({\rho\ke\over 4},x\ke))}^2+\|U\ke^+(1-\psi\ke^+)\|_{\L(B\ke^+\setminus\B({\rho\ke\over 4},x\ke))}^2
		\\
		\le C\left(\|\nabla U\ke^+\|_{\L(B\ke^+\setminus\B({\rho\ke\over 4},x\ke)))}^2+ \rho\ke^{-2}\|U\ke^+\|_{\L(B\ke^+\setminus\B({\rho\ke\over 4},x\ke))}^2\right)+\|U\ke^+\|_{\L(B\ke^+\setminus\B({\rho\ke\over 4},x\ke))}^2\\
		\leq
		\left.
		\begin{cases}
			{Cd\ke^{2(n-2)}\over \rho\ke^{n-2}},&n\geq 3\\
			{C|\ln \rho\ke|^2\over |\ln d\ke|^2},&n=2
		\end{cases}\right\} 
		\leq C_1\rho\ke^n\varkappa\ke^2\gamma\ke^2.\label{We:3}
	\end{multline}
	By Lemma~\ref{lemma:H2} $g^\pm\in\H^2(O^\pm)$, whence $\nabla g^\pm\in \H^1(O^\pm)$. 
	Using the estimate \eqref{trace} (applied for $\nabla g^\pm$), \eqref{assump:2} and the trace estimate 
	$\|\nabla g^\pm\|^2_{\L(\Gamma)}\leq C\| g^\pm\|^2_{\H^2(O^\pm)}$, we obtain
	\begin{align}\notag
		\suml_{k\in\N}\rho\ke\gamma\ke\|\nabla g^\pm\|^2_{\L(B\ke^\pm)} 
		&\leq
		C\suml_{k\in\N}\rho\ke\gamma\ke\Big(\rho\ke \|\nabla g^\pm\|^2_{\L(S\ke )}+
		\rho\ke^2\sum_{i,j=1}^n\left\|{\partial g^\pm\over\partial {x^i}\partial {x^j}}  \right\|^2_{\L(B\ke^\pm)}\Big)\\\notag
		&\leq
		C\sup_{k\in\N}(\rho\ke^2\gamma\ke)\Big(\|\nabla g^\pm\|^2_{\L(\Gamma)}+
		\sum_{i,j=1}^n\left\|{\partial g^\pm\over\partial {x^i}\partial {x^j}}  \right\|^2_{\L(O^\pm)}\Big)
		\\
		&\leq C_1\sup_{k\in\N}(\rho\ke^2\gamma\ke)\|g^\pm\|_{\H^2(O^\pm)}^2.\label{gradg}
	\end{align}
	Using \eqref{mean:mean}, \eqref{lemma:prop1:3}, \eqref{trace:standard}, \eqref{We:1}--\eqref{gradg} and taking into account \eqref{HHH},  \eqref{rhoga:1}, \eqref{Ake:prop}, and
	\begin{gather}\label{rhoga:2}
		\rho\ke\gamma\ke^{1/2}=\eta\ke \rho\ke^{1/2} A\ke,\text{ where }
		 A\ke\text{ is given in \eqref{rhoga:1}}
	\end{gather} 
	(the above equality follows directly from \eqref{rhoga:1}), 
	we get
	\begin{align}\notag 
		\|(\mathscr{W}\e g)^+\|_{\H^1(\Omega^+)}&\leq
		C\Bigl(\suml_{k\in\N}\rho\ke \gamma\ke\|\nabla g^+\|^2_{\L(B\ke^+)}\Bigr)^{1/2}+
		C\Bigl(\suml_{k\in\N}\rho\ke \gamma\ke\|\nabla g^-\|^2_{\L(B\ke^-)}\Bigr)^{1/2}\\\notag 
		&+C\Bigl(\suml_{k\in\N}\varkappa\ke^2\gamma\ke^2\big(\|g^+\|^2_{\L(B\ke^+)}+\|g^-\|^2_{\L(B\ke^-)}\big)\Bigr)^{1/2}
		\\\notag &\leq
		C\sup_{k\in\N}(\rho\ke \gamma\ke^{1/2})\left(\|g^+\| _{\H^2(O^+)}+
		 \|g^-\| _{\H^2(O^-)}\right)\\  &+
		C\sup_{k\in\N}(\rho\ke^{1/2} \varkappa\ke\gamma\ke)\left(\|g^+\| _{\H^1( \Omega^+)}+
		 \|g^-\| _{\H^1( \Omega^-)}\right) 
		 \leq C_1\mu\e\|g\|_{\HS^2}.\label{We:est+}
	\end{align}	
	Similarly,
	\begin{align}
		\|(\mathscr{W}\e g)^-\|_{\H^1(\Omega^-)}\leq C\mu\e\|g\|_{\HS^2}.\label{We:est-}
	\end{align}
	It follows from  \eqref{th2:proof1}--\eqref{Ve:est:2}, \eqref{We:est+}, \eqref{We:est-} and  $\|g\|_{\HS^1}\leq\|g\|_{\HS^2}=\|f\|_{\HS}$  that
	\begin{gather*}
		\|(\A\e+\Id)^{-1}f -(\Id+\mathscr{K}\e)(\A+\Id)^{-1} f\|_{\H^1(\Omega\setminus\Gamma)}\leq C\mu\e\|f\|_{\HS}.
	\end{gather*}
	
	It remains to prove \eqref{K:prop1}--\eqref{K:prop2}.
	Let $f\in\HS$ and $g=(\A+\Id)^{-1}f$.
	Using \eqref{assump:1}, \eqref{lemma:prop1:4}--\eqref{trace:standard} and the Cauchy-Schwarz inequality, we get
	\begin{align*}\notag
		\|(\mathscr{K}\e f)^\pm\|^2_{\L(\Omega^\pm)}&=
		{1\over 4}\suml_{k\in\N} |\la [g] \ra_{S\ke}|^2\\
		&\notag\leq C
		\sum_{k\in\N}\|U\ke^\pm\|^2_{\L(B\ke^\pm)}\rho\ke^{2}\gamma\ke\left(\|g^+\|^2_{\L(S\ke)}+\|g^-\|^2_{\L(S\ke)}\right)
		\\\notag
		&\leq C_1 \sup_{k\in\N}(\rho\ke^{2}\gamma\ke  )\|g\|_{\HS^1}^2
		\leq C_1 \sup_{k\in\N}(\rho\ke^{2}\gamma\ke  )\|g\|_{\HS^2}^2\\\notag
		&=C_1\ \sup_{k\in\N}(\rho\ke^{2}\gamma\ke  )\|f\|_{\HS}^2\leq
		C_2\ \sup_{k\in\N}(\eta\ke^{2}\rho\ke  )\|f\|_{\HS}^2 
	\end{align*}
	(on the last step we use \eqref{rhoga:2} and \eqref{Ake:prop}).
	Furthermore, using \eqref{capty:eq} and \eqref{H:symm}, we obtain
	\begin{gather*} 
		\|\nabla (\mathscr{K}\e f)^\pm\|^2_{\L(\Omega^\pm)}=
		{1\over 4}\suml_{k\in\N}\|\nabla U\ke^\pm\|^2_{\L(B\ke^\pm)}|\la [g] \ra_{S\ke}|^2=
		{1\over 2}\left(\|\ga^{1/2}[g]\|^2_{\L(\Gamma)}+  L\e\right), 
	\end{gather*}	
	where $L\e\ceq {1\over 4}\sum_{k\in\N}\capty(D\ke)|\la [g] \ra_{S\ke}|^2-\|\ga^{1/2}[g]\|^2_{\L(\Gamma)} .$
	Using \eqref{assump:main}, the trace theorem and Lemma~\ref{lemma:H2}, we get the required estimate on the remainder $L\e$:
	\begin{align*}
		|L\e|&\leq  \kappa\e\|[g]\|_{\H^{3/2}(\Gamma)}\|[g]\|_{\H^{1/2}(\Gamma)}
		\leq
		C\kappa\e\left(\|g\|_{\H^2(O^+)}+\|g\|_{\H^2(O^-)}\right)\left(\|g\|_{\H^1(\Omega^+)}+\|g\|_{\H^1(\Omega^-)}\right)\\
		&\leq C_1\kappa\e\|g\|_{\HS^2}\|g\|_{\HS^1}\leq
		C_1\kappa\e\|g\|^2_{\HS^2}=C_1\kappa\e\|f\|_{\HS}^2.
	\end{align*}
	Theorem~\ref{th2} is proven. 
	
	\begin{remark}
		\label{rem:intersection}
		Lemma~\ref{lemma:H2} is   the only place  where 
		we use the special form of the domain $\Omega$, see \eqref{h}. If $\Gamma$ intersects $\Omega$, we face a problem -- $\H^2$-regularity fails  in a neighborhood of $\partial\Omega\cap\Gamma$, where $\delta'$-interface conditions ``meet'' with the Neumann boundary conditions.
		This difficulty can be overcame in the case when the holes belong to a subset $\Gamma'\subset\Gamma\cap\Omega$ with $\dist(\Gamma',\partial\Omega)>0$ and $\supp(\gamma)\subset\Gamma'$ (cf. Remark~\ref{rem:Gamma'}). In this case the functions from $\dom(\A)$ belong to $\H^2(O'^\pm)$, where
		$O'^\pm\ceq\{x=(x',x^n):\  x'\in \Gamma',\ |x^n|<\dist(\Gamma',\partial\Omega)/2,\ \pm x^n>0\}$, and the estimate \eqref{H2est} holds with $O^\pm$ being replaced by $O'^\pm$ -- these properties allows to carry out the proof in the same was as it is done above for $\Omega$ satisfying \eqref{h}.

	\end{remark}

	\section{Examples\label{sec4} }
	
	In this section we present two examples for which the assumption \eqref{assump:main} holds true.
	
	\subsection{Very small holes}
	The first example deals with the case when the holes $D\ke$ are so small that 
	the limiting function $\gamma$ is zero. 
	
	Assume that the sets $D\ke\subset\Gamma$ satisfies the conditions \eqref{assump:1} and \eqref{assump:2},
	while \eqref{assump:3} is complemented by the stronger assumption
	\begin{gather*}
		\lim_{\eps\to 0}\sup_{k\in\N}\gamma\ke\to 0.
	\end{gather*}
	Let $g\in\H^{3/2}(\Gamma)$ and $h\in \H^{1/2}(\Gamma)$. 
	Using the Cauchy-Schwarz inequality, the estimate \eqref{capty:bound} and
	taking into account that $\area(S\ke)=C\rho\ke^{n-1}$, we obtain
	\begin{align}\notag
		\left|\suml_{k\in\N}\capty(D\ke)
		\la g \ra_{S\ke}\la \overline{h} \ra_{S\ke}\right|
		&\leq
		\left(\suml_{k\in\N} \capty(D\ke)|\la g \ra_{S\ke}|^2\right)^{1/2} 
		\left(\suml_{k\in\N}\capty(D\ke)|\la \overline{h} \ra_{S\ke}|^2\right)^{1/2}
		\\\notag 
		&\leq
		\left(\suml_{k\in\N} {\capty(D\ke)\over \area(S\ke)}\|g\|^2_{\L(S\ke)} \right)^{1/2} 
		\left(\suml_{k\in\N} {\capty(D\ke)\over \area(S\ke)}\|h\|^2_{\L(S\ke)}\right)^{1/2}
		\\\notag 
		&\leq C \sup_{k\in\N}\gamma\ke ^{1/2}\|g\|_{\L(\Gamma)}\|h\|_{\L(\Gamma)}
		\leq C \sup_{k\in\N}\gamma\ke ^{1/2}\|g\|_{\H^{3/2}(\Gamma)}\|h\|_{\H^{1/2}(\Gamma)}.
	\end{align}
	Thus    \eqref{assump:main} holds with $\gamma\equiv 0$ and $\kappa\e=C \sup_{k\in\N}\gamma\ke ^{1/2}$.

	\subsection{Regularly distributed holes}
	In the second case we consider the holes whose location is subordinate to 
	some partition of $\Gamma$ and show that their sizes can be chosen such a way
	that \eqref{assump:main} holds with any predefined $\gamma$.
	
	For simplicity, we assume that $n\ge 3$. Furthermore, we assume that the holes $D\ke\subset\Gamma$ are balls, i.e. $D\ke=\B(d\ke,\rho\ke)\cap\Gamma$. 
	As before, we assume that there exist a sequence $\left\{\rho\ke,\ k\in \N\right\}$ of positive numbers such that \eqref{assump:1}, \eqref{assump:2} hold.
	Additionally, we assume that there exists  a family $\left\{\Gamma\ke,\ k\in \N\right\}$ of relatively open  subsets of $\Gamma$ 
	such that  
	\begin{align}
		\label{Gammake:cond1}
		&\overline{B\ke\cap \Gamma}\subset\Gamma\ke\quad \text{(recall that $B\ke= \B(\rho\ke,x\ke)$)},
		\\
		\label{Gammake:cond2}
		&\Gamma_{k,\varepsilon}\cap \Gamma_{l,\varepsilon}=\emptyset\text{ if }k\not=l,
		\\
		\label{Gammake:cond3}
		&
		\cup_{k\in\N}\overline{\Gamma\ke}=\Gamma,
		\\
		&\label{Gammake:cond4}
		\exists m\in\N\ \forall x\in\Gamma:\  \#(\{k\in\N:\ x\in\conv(\Gamma\ke)\})\leq m,
		\\
		&
		\label{r:cond}
		{r}\ke\leq C\rho\ke,
	\end{align}
	where ${r}\ke\ceq\diam(\Gamma\ke)$, by
	$\conv(\Gamma\ke)$ we denote   the convex hull of $\Gamma\ke$, and $\#(\dots)$ stands for the cardinality of the enclosed set (that is, if all  $\Gamma\ke$ are convex, then \eqref{Gammake:cond4} holds with $m=1$).
	The above conditions are illustrated on Figure~\ref{fig2}.
	
	\begin{figure}[h]
		\begin{picture}(400,150) \centering
			\scalebox{1}{
				\includegraphics{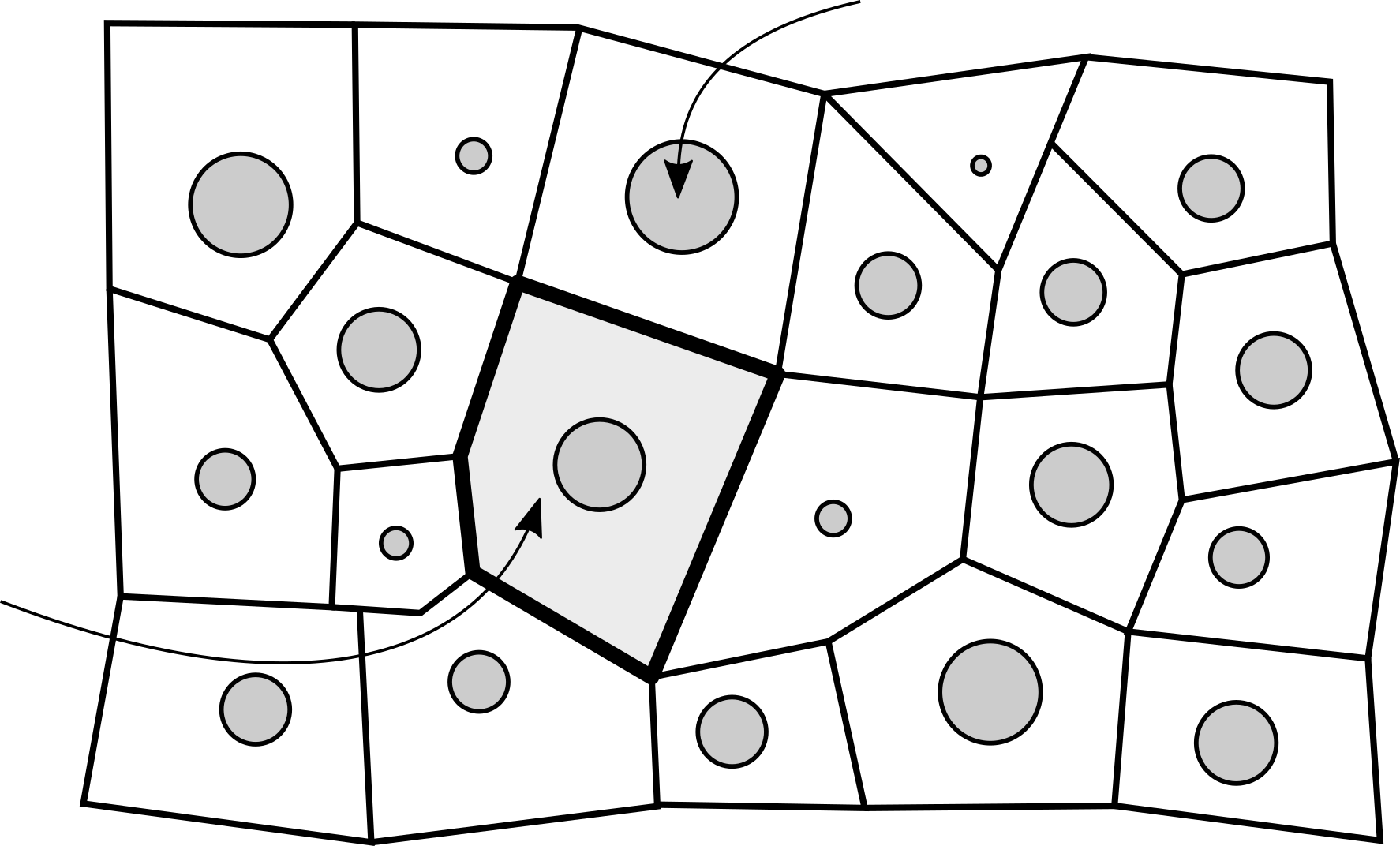}\qquad\qquad\qquad
				\includegraphics{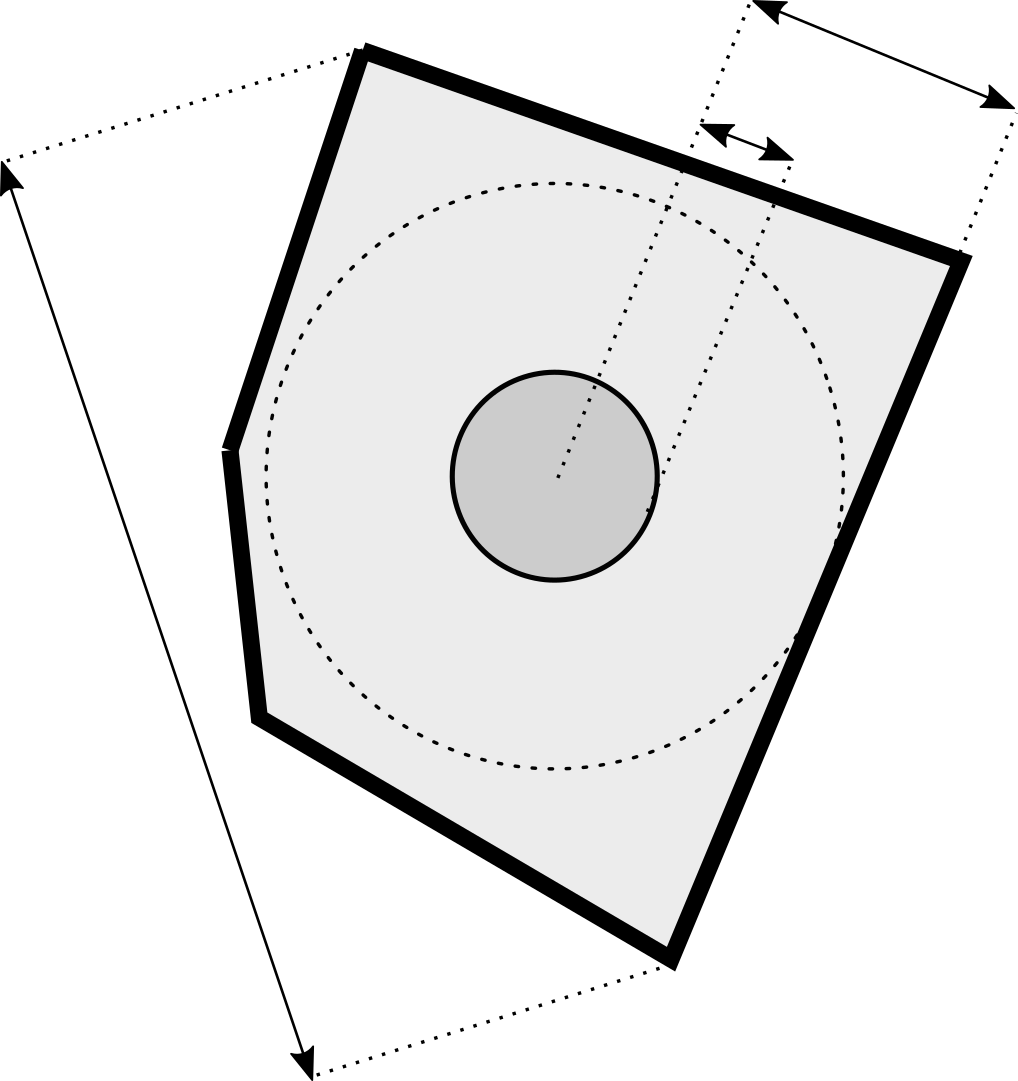}
				
				\put(-420,140){\emph{(a)}}
				\put(-130,140){\emph{(b)}}

				\put(-419,38){$\Gamma\ke$}
				\put(-269,127){$D\ke$}
				
				\put(-115,45){\rotatebox{20}{$^{r\ke}$}}
				\put(-21,131){\rotatebox{-20}{$_{\rho\ke}$}}
				\put(-37,115){\rotatebox{-20}{$^{d\ke}$}}}
			
		\end{picture}
		\caption{\emph{(a)} The ``sieve'' $\Gamma\e$ and a possible choice of the sets $\Gamma\ke$ satisfying \eqref{Gammake:cond1}--\eqref{Gammake:cond3}. 
			\emph{(b)} The set $\Gamma\ke$ and its relation with the numbers $d\ke$, $\rho\ke$ and $r\ke$. 
		}\label{fig2}
	\end{figure}
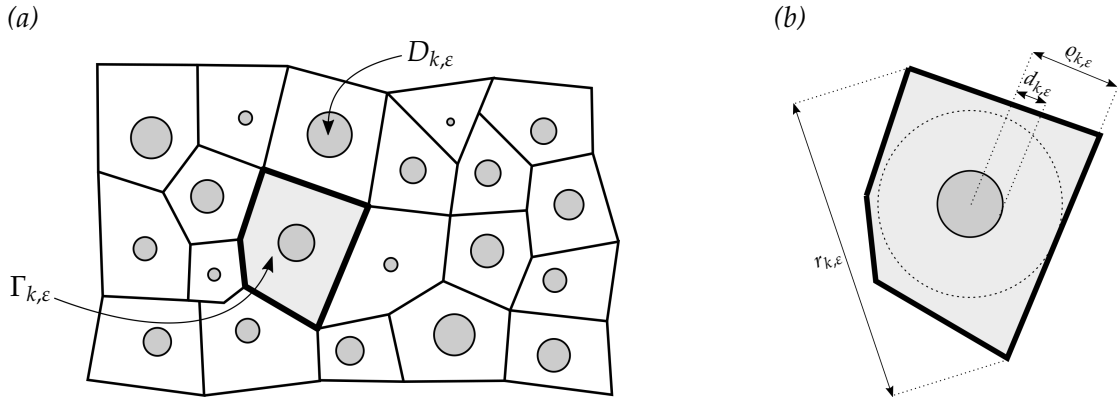 	
	
	We specify the numbers $d\ke$. Let $\gamma\in C^1(\Gamma)\cap\mathsf{W}^{1,\infty}(\Gamma)$ be arbitrary positive function. Set
	\begin{gather}\label{dke}
		d\ke=(4\al^{-1} \gamma(x\ke)\area(\Gamma\ke))^{1/(n-2)},
	\end{gather}
	where the constant $\al$ is the Newtonian capacity of the set $\B(1,0)\cap\Gamma $, i.e.
	\begin{gather*}
		\al=\inf_U\|\nabla U\|_{\L(\R^n)}^2,
	\end{gather*}
	where the infimum is taken over $U\in  {C}_0^\infty(\R^n)$ being
	equal to $1$ on a neighbourhood of $\B(1,0)\cap\Gamma$ (with the neighbourhood depending upon $U$).
	Due to \eqref{r:cond} (which implies $\area(\Gamma\ke)\leq C\rho\ke^{n-1}$)  this
	choice of $d\ke$ is indeed  admissible, i.e. the condition \eqref{assump:3} is  fulfilled.
	\smallskip
	
	We are in position to formulate the main result of this subsection.
	
	\begin{proposition}\label{prop:kappae}
		Let the assumptions \eqref{assump:1}, \eqref{assump:2}, \eqref{Gammake:cond1}--\eqref{dke} hold true. Then the condition \eqref{assump:main} is fulfilled with $\kappa\e=C\sup_{k\in\N}\rho\ke^{1/2}$.
	\end{proposition}
	
	\begin{proof}
		First, we introduce several new sets:	
		\begin{itemize} 	
			\item $\wh\Gamma\ke\ceq\conv(\Gamma\ke)$,
			\item $ P\ke{{}}\ceq  \left\{x=(x',x^n)\in\R^n:\ x'\in \Gamma\ke,\ 0<|x^n|<{\rho}\ke,\  x^n>0\right\}$,
			\item $\wh P\ke{{}}\ceq  \left\{x=(x',x^n)\in\R^n:\ x'\in \wh\Gamma\ke,\ 0<|x^n|<{\rho}\ke,\  x^n>0\right\}$,
		\end{itemize}
		Due to \eqref{eps0:1+}, one has $\wh P\ke\subset O^+$. 
		Furthermore,  by virtue of \eqref{Gammake:cond4}, one has the  estimate
		\begin{gather}\label{m:est}
			\forall f\in\L(O^+):\ \sum_{k\in\N}\|f\|^2_{\L(\wh P\ke)}\leq  m\|f\|^2_{\L(\cup_{k\in\N}\wh P\ke)} \leq m \|f\|^2_{\L(O^+)}.
		\end{gather}
		
		\begin{lemma}
			\label{lemma:aux}
			One has
			\begin{align}\label{lemma:aux:1}
				\forall v\in \H^1(\wh P\ke{{}}):&&
				\|v\|^2_{\L(\wh\Gamma\ke)}\leq
				2\rho\ke^{-1}\|v\|^2_{\L(\wh P\ke{{}})}+2\rho\ke\|\nabla v\|^2_{\L(\wh P\ke{{}})} ,\\
				\label{lemma:aux:2}
				\forall v\in \H^1(\wh P\ke{{}}):&&
				\|v-\langle v \rangle_{\wh P\ke}\|^2_{\L(\wh P\ke)}\leq
				C \rho\ke^2 \|\nabla v\|^2 _{\L(\wh P\ke{{}})}.
			\end{align}
		\end{lemma}
		
		\begin{proof}
			It is enough to proof \eqref{lemma:aux:1} for $v\in C^\infty(\overline{\wh P\ke{{}}})$.
			Let $ (x',0)\in\wh\Gamma\ke$ and $(x',z)\in\wh P\ke$ with $z\in (0,\rho\ke)$.
			One has
			\begin{gather*}
				v(x',0)=v(x',z)-\int_{0}^{z}{\partial v\over\partial x^n}(x',\tau)\d\tau,    
			\end{gather*}
			whence, using the Cauchy-Schwarz inequality, we get
			\begin{gather*}
				|v(x',0)|^2\leq 2|v(x',z)|^2+2\rho\ke\int_{0}^{\rho\ke}\left|{\partial v\over\partial x^n}(x',\tau)\right|^2\d\tau\leq
				2|v(x',z)|^2+2\rho\ke\int_{0}^{\rho\ke}\left|\nabla v(x',\tau)\right|^2\d\tau.
			\end{gather*}
			Integrating this inequality over $z\in (0,\rho\ke)$ and over $x'\in\wh\Gamma\ke$, then dividing by $\rho\ke$, 
			we arrive at the  estimate \eqref{lemma:aux:1}. 
			
			To prove the estimate \eqref{lemma:aux:2} we need the following result \cite[Theorem~3.2]{Be03}:
			let $D$ be a convex domain, then 
			$$\|v-\langle v \rangle_{D}\| _{\L(D)}\leq
			{\diam(D)\over \pi} \|\nabla v\| _{\L(D)},\ \forall v\in \H^1(D).$$
			Applying it for $D= \wh P\ke$ (evidently, this set is complex)   and taking into account 
			(cf.~\eqref{r:cond})
			that $\diam(\wh P\ke)\leq C\rho\ke$,  
			we immediately get the required estimate \eqref{lemma:aux:2}. The lemma is proven.
		\end{proof}

		\begin{lemma}\label{lemma:<>}
			One has
			\begin{align}
				\label{est<>2}
				\forall v\in \H^1(\wh P\ke{{}}):&\quad \left|\langle v \rangle_{\Gamma\ke{{}}}-\langle v \rangle_{\wh P\ke{{}}}\right|^2
				\leq {C_1 {\rho}\ke^{2-n}}
				\|\nabla v\|^2_{\L(\wh P\ke{{}})}. 
				\\
				\label{lemma:<>:est}
				\forall v\in \H^1(\wh P\ke{{}}):&\quad
				\left|\langle v \rangle_{\Gamma\ke}-\langle v \rangle_{S\ke}\right|^2\leq
				C{\rho}\ke^{2-n}\|\nabla v\|^2_{\L(\wh P\ke{{}})}.
			\end{align}
		\end{lemma}
		
		\begin{proof}
			Using the Cauchy-Schwarz inequality, the estimates \eqref{lemma:aux:1}--\eqref{lemma:aux:2}, 
			and taking into account that $\area(\Gamma\ke)\ge C{\rho}\ke^{n-1}$,
			we get
			\begin{align}\notag
				&
				\left|\langle v \rangle_{\Gamma\ke{{}}}-\langle v \rangle_{\wh P\ke{{}}}\right|^2=
				\left|\langle v -\langle v \rangle_{\wh P\ke{{}}}\rangle_{\Gamma\ke{{}}}\right|^2\\\notag
				&\qquad
				\leq
				{1\over \area(\Gamma\ke)}\|v-\langle v \rangle_{\wh P\ke{{}}}\|^2_{\L(\Gamma\ke)}\leq
				{1\over \area(\Gamma\ke)}\|v-\langle v \rangle_{\wh P\ke{{}}}\|^2_{\L(\wh \Gamma\ke)}
				\\
				&\qquad
				\leq {C\over \area(\Gamma\ke)}
				\left({\rho}\ke^{-1}\|v-\langle v \rangle_{\wh P\ke{{}}}\|^2_{\L(\wh P\ke{{}})}+
				{\rho}\ke\|\nabla v\|^2_{\L(\wh P\ke{{}})}\right) 
				\leq {C_1 {\rho}\ke^{2-n}}
				\|\nabla v\|^2_{\L(\wh P\ke{{}})}.\label{est<>2+}
			\end{align}	
			To proceed further we need an auxiliary	result from \cite{Kh09}:	
			let $D\subset\R^n$ be a bounded convex domain, $D_1$, $D_2$ be arbitrary measurable subsets of $D$ with $\vol(D_1)\not=0$, $\vol(D_2)\not=0$, then
			\begin{gather*}
				\forall v\in \H^1(D):\quad \left|\langle v \rangle_{D_1}-\langle v \rangle_{D_2}\right|^2\leq
				C{(\mathrm{diam}(D))^{n+2}\over \vol(D_1)\cdot \vol(D_2)}\|\nabla v\|^2_{\L(D)},
			\end{gather*}
			where  the constant $C>0$ depends only on $n$.	
			Applying this estimate for $D\ceq \wh P\ke{{}}$, $D_1\ceq \wh P\ke{{}}$ and $D_2\ceq B\ke^+$
			and taking into account that $\diam(\wh P\ke )\leq C{r}\rho\ke$ and $\vol(\wh P\ke{{}})\geq \vol(B\ke^+)=C\rho\ke^{n}$, we
			obtain
			\begin{gather}\label{est<>1}
				\forall v\in \H^1(\wh P\ke{{}}):\quad 
				\left|\langle v \rangle_{\wh P\ke{{}}}-\langle v \rangle_{B\ke^+}\right|^2\leq
				C{\rho}\ke^{2-n}\|\nabla v\|^2_{\L(\wh P\ke{{}})}.    
			\end{gather}
			Combining \eqref{est<>2+}, \eqref{est<>1} and \eqref{mean:mean}, we arrive at the   estimate 
			\eqref{lemma:<>:est}. The lemma is proven.
		\end{proof}
		
		\begin{lemma}
			One has:
			\begin{gather}\label{alalal}
				\al d\ke^{n-2}\leq \capty(D\ke) \leq    \al d\ke^{n-2}(1+C\rho\ke). 
			\end{gather}
		\end{lemma}
		
		\begin{proof}
			Let $\al\e$ be the  Newtonian capacity of the set $D\ke=\B(d\ke,x\ke)\cap\Gamma $,
			i.e.
			\begin{gather}\label{ale:def}
				\al\e=\inf_U\|\nabla U\|_{\L(\R^n)}^2,
			\end{gather}
			where the infimum is taken over $U\in  {C}_0^\infty(\R^n)$ being
			equal to $1$ on a neighbourhood of $D\ke$.
			Simple re-scaling arguments yield
			\begin{gather}\label{al:al}
				\al\e = \al d\ke^{n-2}.    
			\end{gather}
			
			Let $U\ke$ be the solution to the problem \eqref{BVP:cap} extended by zero to the whole $\R^n$.
			For each $\delta>0$ there exists $U^\delta\ke\in C^\infty_0(\R^n)$ such that
			$\|\nabla U\ke - \nabla U\ke^\delta\|_{\L(\R^n)}\leq \delta $ and 
			$U\ke^\delta=1$ on a neighborhood of $D\ke$. 
			Using \eqref{ale:def} and \eqref{capty:eq},
			we get  
			\begin{gather*}
				\al\e^{1/2} \leq \|\nabla U\ke^\delta\|_{\L(\R^n)}\leq
				 \|\nabla U\ke\|_{\L(\R^n)}+\delta=\sqrt{\capty(D\ke)}+\delta,
			\end{gather*}
			whence, sending $\delta\to0$, we obtain
			\begin{gather}\label{al:al+}
				\al\e  \leq  \capty(D\ke).
			\end{gather}
		    Combining \eqref{al:al} and \eqref{al:al+} we get the first estimate in \eqref{alalal}.
			
			To prove the second estimate   we 
			introduce the function $V\ke$ -- the solution to the problem
			\begin{gather*}
				\begin{cases}
					\Delta V=0,&x\in \R^n\setminus\overline{D\ke},\\
					V=1,&x\in\partial D\ke =\overline{D\ke},\\
					V\to 0,&|x|\to \infty.
				\end{cases}    
			\end{gather*} 
			It is well-known that 
			\begin{gather*}
				\al\e=\|\nabla V\ke\|^2_{\L(\R^n)}.   
			\end{gather*}
			Recall that the function $\psi\ke$ is defined in \eqref{phipsi}.
			Using \eqref{capty:inf} and taking into account that
			$ V\ke\psi\ke\in\H^1_0(B\ke)$ and $(V\ke\psi\ke)\restriction_{\overline{D\ke}}=1 $, we get
			\begin{gather}\label{cap<al}
				\capty(D\ke) \leq \|\nabla (V\ke\psi\ke)\|^2_{\L(B\ke)}= \|\nabla (V\ke\psi\ke)\|^2_{\L(\R^n)}=
				\|\nabla V\ke\|^2_{\L(\R^n)}+R\e=\al\e + R\e,
			\end{gather}    
			where  $R\e=R\e^1+R\e^2$ with
			$$
			R\e^1\ceq \|\nabla (V\ke(\psi\ke-1))\|^2_{\L(\R^n)},\quad R\e^2\ceq2(\nabla V\ke,\nabla (V\ke(\psi\ke-1)))_{\L(\R^n)}.
			$$
			The function $V\ke$ obeys the following estimates:
			\begin{gather}\label{VV:ests}
				\forall x\in \R^n\setminus
				\B({\rho\ke\over 4},x\ke):\quad |V\ke(x)|\leq Cd\ke^{n-2}|x-x\ke|^{2-n},\quad
				|\nabla V\ke(x)|\leq Cd\ke^{n-2}|x-x\ke|^{1-n}.
			\end{gather}
			(the proof is similar to the proof of \eqref{H:estimates:1}--\eqref{H:estimates:2}, an alternative proof can be found in \cite[Lemma~2.1]{MK06}).
			Using \eqref{VV:ests} and  the properties
			\begin{gather*}
			0\leq \psi\ke\leq 1,
			\quad 
			|\nabla\psi\ke|\leq C\rho\ke^{-1},
			\\
			\supp(1-\psi\ke)  \subset  {\R^n\setminus
				\B({\rho\ke\over 4},x\ke)},\quad
			\supp(\nabla\psi\ke)\subset\overline{B\ke\setminus \B({\rho\ke\over 4},x\ke)},
			\end{gather*}
			we obtain
			\begin{gather}
				\label{R1}
				|R\e^1|\leq 
				C\left(\|\nabla V\ke\|^2_{\L(\R^n\setminus
					\B({\rho\ke\over 4},x\ke)))}+
				\rho\ke^{-2}\|V\ke\|^2_{\L(B\ke\setminus
					\B({\rho\ke\over 4},x\ke))}\right)
				\leq C_1 d\ke^{2(n-2)}\rho\ke^{2-n}.
			\end{gather}
			Similarly,
			\begin{gather}
				\label{R2}
				|R\e^2|\leq 
				C\|\nabla V\ke\|^2_{\L(\R^n\setminus \B({\rho\ke\over 4},x\ke))}\sqrt{|R\e^1|}
				\leq C_1 d\ke^{2(n-2)}\rho\ke^{2-n}.
			\end{gather}
			Combining \eqref{R1}--\eqref{R2} and taking into account that 
			$d\ke^{n-2}\leq C\rho\ke^{n-1}$, we finally get
			\begin{gather}
				\label{R}
				R\e\leq Cd\ke^{2(n-2)}\rho\ke^{2-n}\leq C_1d\ke^{n-2}\rho\ke.
			\end{gather}
			The desired estimate follows from \eqref{al:al}, \eqref{cap<al}, \eqref{R}.
			The lemma is proven.
		\end{proof}
		
		Now, we proceed to the proof of Proposition~\ref{prop:kappae}.	
		Let 
		$g\in \H^{3/2}(\Gamma)$ and
		$h\in \H^{1/2}(\Gamma)$. 
		Taking into account \eqref{Gammake:cond2}--\eqref{Gammake:cond3} one can represent the left-hand-side of \eqref{assump:main} as follows:
		\begin{gather}\label{P1-5}
			{1\over 4}\suml_{k\in\N}\capty(D\ke)
			\la g \ra_{S\ke}\la \overline{h} \ra_{S\ke}-(\gamma g,h)_{\L(\Gamma)}=
			\sum_{k=1}^5 P\e^k,
		\end{gather}
		where
		\begin{align*}
			P\e^1&\ceq
			{1\over 4}\suml_{k\in\N}\capty(D\ke)
			\la g \ra_{S\ke}\big(\la \overline{h} \ra_{S\ke}-\la \overline{h} \ra_{\Gamma\ke}\big),\\
			P\e^2&\ceq
			{1\over 4}\suml_{k\in\N}\capty(D\ke)
			\big(\la g \ra_{S\ke}-\la g \ra_{\Gamma\ke}\big)\la \overline{h} \ra_{\Gamma\ke},
			\\
			P\e^3&\ceq\suml_{k\in\N}\left({1\over 4}\capty(D\ke)
			-\gamma(x\ke)\area(\Gamma\ke)\right)\la g\ra_{\Gamma\ke}\la \overline{h} \ra_{\Gamma\ke},
			\\
			P\e^4&\ceq\suml_{k\in\N}
			\int_{\Gamma\ke} \gamma(x\ke)g(x)\left(\la \overline{h} \ra_{\Gamma\ke}- \overline{h(x)}\right)\d x',
			\\
			P\e^5&\ceq
			\suml_{k\in\N}
			\int_{\Gamma\ke}\left(\gamma(x\ke) - \gamma(x)\right)g(x)\overline{h(x)}\d x'.
		\end{align*}

		It is well-know (see, e.g., \cite[Theorem~3.1]{LP20}) that there exists a bounded linear operator $\mathscr{E}:\H^{1/2}(\Gamma)\to\H^1(O^+)$, which is a right inverse of the trace operator
		$\H^1(O¨^+)\to  \H^1(\Gamma)$.
		Then, using the Cauchy-Schwarz inequality,  \eqref{lemma:<>:est}, \eqref{m:est}
		and the fact (cf.~\eqref{capty:bound}) that 
		$$
		{\capty(D\ke)}\leq C\rho\ke^{n-1},\quad \area(S\ke)=C_1\rho\ke^{n-1},\quad \area(\Gamma\ke)\geq C\rho\ke^{n-1},  
		$$
		we get
		\begin{align}\notag
			|P\e^1|&\leq {1\over 4}
			\left(\suml_{k\in\N}\capty(D\ke)|\la g \ra_{S\ke}|^2\right)^{1/2}
			\left(\suml_{k\in\N}\capty(D\ke)\left|\la \overline{h} \ra_{S\ke}-\la \overline{h} \ra_{\Gamma\ke}\right|^2\right)^{1/2}
			\\\notag
			&\leq C\left(\suml_{k\in\N}\|g\|^2_{\L(S\ke)}\right)^{1/2}
			\left(\suml_{k\in\N}\rho\ke\|\nabla (\mathscr{E}h)\|^2_{\L(\wh P\ke)} \right)^{1/2}
			\\ 
			&\leq C_1\sup_{k\in\N}\rho\ke^{1/2}
			\|g\|_{\L(\Gamma)}\|\nabla (\mathscr{E}h)\|_{\L(\Omega^+)}
			\leq C_2\sup_{k\in\N}\rho\ke^{1/2}
			\|g\|_{\L(\Gamma)}\|h\|_{\H^{1/2}(\Gamma)},
			\label{P1}
		\end{align}
		and, similarly,  
		\begin{gather}\label{P2}
			|P\e^2|\leq	C\sup_{k\in\N}\rho\ke^{1/2}
			\|g\|_{\H^{1/2}(\Gamma)}\|h\|_{\L(\Gamma)}.
		\end{gather}
		Further, by virtue of \eqref{dke}, \eqref{alalal} and \eqref{assump:3} one has 
		$$\left|{1\over 4}\capty(D\ke)
		-\gamma(x\ke)\area(\Gamma\ke)\right|= {1\over 4}\left|\capty(D\ke)-\al d\ke^{n-2}\right|\leq Cd\ke^{n-2}\rho\ke\leq C_1\rho\ke^{n},$$ 
		whence, taking into account that $\area(\Gamma\ke)\geq C\rho\ke^{n-1}$, we get
		\begin{gather}
			|P\e^3|\leq C\left(\sum_{k\in\N}\rho\ke^n|\la g\ra_{\Gamma\ke}|^2\right)^{1/2}
			\left(\sum_{k\in\N}\rho\ke^n|\la h\ra_{\Gamma\ke}|^2\right)^{1/2}\leq
			C\sup_{k\in\N}\rho\ke\|g\|_{\L(\Gamma)}\|h\|_{\L(\Gamma)}.\label{P3}
		\end{gather}
		Using \eqref{m:est}--\eqref{est<>2}, we obtain
		\begin{align}\notag
			|P\e^4|&\leq C\left(\sum_{k\in\N}\|g\|^2_{\L(\Gamma\ke)}\right)^{1/2}
			\left(\sum_{k\in\N}\|h-\la h\ra_{\Gamma\ke}\|^2_{\L(\Gamma\ke)}\right)^{1/2}
			\\\notag
			&\leq C_1\|g\|_{\L(\Gamma)}\left(\sum_{k\in\N}\left(\rho\ke^{-1}
			\|\mathscr{E}h-\la h\ra_{\Gamma\ke}\|^2_{\L(\wh P\ke)}+ \rho\ke\|\nabla \mathscr{E}h\|^2_{\L(\wh P\ke)}\right)\right)^{1/2}
			\\\notag
			&\leq C_1\|g\|_{\L(\Gamma)}\Bigg(\sum_{k\in\N}\Big(2\rho\ke^{-1}
			\|\mathscr{E}h-\la \mathscr{E} h\ra_{\wh P\ke}\|^2_{\L(\wh P\ke)}+
			2\rho\ke^{-1}|\la \mathscr{E}h\ra_{\wh P\ke}-\la h\ra_{\Gamma\ke}|^2\vol(\wh P\ke)+
			\\\notag
			& \rho\ke\|\nabla \mathscr{E}h\|^2_{\L(\wh P\ke)}\Big)\Bigg)^{1/2}
			\\\notag
			&\leq C_2\|g\|_{\L(\Gamma)} \left(\sum_{k\in\N}\rho\ke\|\nabla \mathscr{E}h\|^2_{\L(\wh P\ke)}\right)^{1/2}
			\leq C_3\sup_{k\in\N}\rho\ke^{1/2}\|g\|_{\L(\Gamma)}\|\nabla \mathscr{E}h\|_{\L(O^+)}
			\\
			&\leq C_4\sup_{k\in\N}\rho\ke^{1/2}\|g\|_{\L(\Gamma)}\| h\|_{\H^{1/2}(\Gamma)}.
			\label{P4}
		\end{align} 
		Finally, since $|\gamma(x)-\gamma(x\ke)|\leq C|x-x\ke|$ (recall that $\gamma \in C^{1}(\Gamma)\cap\mathsf{W}^{1,\infty}(\Gamma)$), we infer
		\begin{gather}\label{P5}
			|P\e^5|\leq C\sup_{k\in\N}\rho\ke\suml_{k\in\N}
			\|g\|_{\L(\Gamma)}\|h\|_{\L(\Gamma)}.	
		\end{gather}
		The statement of the proposition follows immediately from \eqref{P1-5}--\eqref{P5}.
	\end{proof}

	\section*{Acknowledgements}
	
	The author gratefully acknowledges financial support by the Czech Science Foundation (GA\v{C}R) through the project 22-18739S and by the research program ``Mathematical Physics and Differential Geometry'' of the Faculty of Science of the University of Hradec Kr\'alov\'e.
	The author thanks Jussi Behrndt and Vladimir Lotoreichik for useful suggestions.
	
	\bibliographystyle{abbrv}
	\bibliography{references}

\end{document}